\documentclass{article}
\usepackage{arxiv}
\usepackage[utf8]{inputenc} 
\usepackage[T1]{fontenc}    
\usepackage{hyperref}       
\usepackage{url}            
\usepackage{booktabs}       
\usepackage{amsfonts}       
\usepackage{nicefrac}       
\usepackage{microtype}      
\usepackage{graphicx}
\usepackage{doi}
\usepackage{amssymb, amsmath, amsthm}
\usepackage{mathrsfs}
\usepackage{tikz}

\newcommand{\squaremi}{{\square_{i}}^{{\kern-2pt}-}}
\newcommand{\lozengemi}{{\lozenge_{i}}^{{\kern-2.5pt}-}}
\newcommand{\PhiIHplus}{{{\Phi}_{I,\mathscr{H}}}^{{\kern-14.5pt}+}\kern9pt}
\newcommand{\PhiIHminus}{{{\Phi}_{I,\mathscr{H}}}^{{\kern-14.5pt}-}\kern9pt}
\newcommand{\PhiIH}{\Phi_{I,\mathscr{H}}}
\newcommand{\PhiIHPF}{{{\Phi}_{I,\mathscr{H}}}^{{\kern-14.5pt}PF}\kern2pt}

\newtheorem{definition}{Definition}
\newtheorem{theorem}{Theorem}
\newtheorem{proposition}{Proposition}
\newtheorem{remark}{Remark}
\newtheorem{lemma}{Lemma}
\newtheorem{example}{Example}

\title{Hennessy-Milner Type Theorems for Fuzzy\\ Multimodal Logics Over Heyting Algebras}

\author{ \href{https://orcid.org/0000-0003-0563-4267}{\includegraphics[scale=0.06]{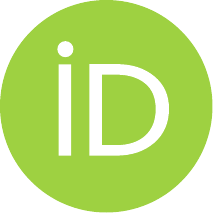}\hspace{1mm}Marko Stankovi\'{c}}\\
	University of Ni\v{s}\\
	Pedagogical Faculty in Vranje\\
	Partizanska 14, 17500 Vranje, Serbia \\
	\texttt{markos@pfvr.ni.ac.rs} \\
	\And
	\href{https://orcid.org/0000-0001-8625-4682}{\includegraphics[scale=0.06]{orcid.pdf}\hspace{1mm}Miroslav \'{C}iri\'{c}} \\
    University of Ni\v{s}\\
    Faculty of Sciences and Mathematics\\
    Vi\v segradska 33, 18000 Ni\v{s}, Serbia\\
	\texttt{miroslav.ciric@pmf.edu.rs} \\
    \And
	\href{https://orcid.org/0000-0002-9853-3993}{\includegraphics[scale=0.06]{orcid.pdf}\hspace{1mm}Jelena Ignjatovi\'{c}} \\
    University of Ni\v{s}\\
    Faculty of Sciences and Mathematics\\
    Vi\v segradska 33, 18000 Ni\v{s}, Serbia\\
	\texttt{jelena.ignjatovic@pmf.edu.rs} \\
}

\renewcommand{\shorttitle}{Hennessy-Milner Type Theorems for Fuzzy Multimodal Logics Over Heyting Algebras}

\hypersetup{
pdftitle={Hennessy-Milner Type Theorems for Fuzzy Multimodal Logics Over Heyting Algebras},
pdfsubject={Fuzzy logic},
pdfauthor={Marko Stankovi\'{c}, Miroslav \'{C}iri\'{c}, Jelena Ignjatovi\'{c}},
pdfkeywords={Fuzzy bisimulation, fuzzy Kripke model, fuzzy multimodal logic, Hennessy-Milner property, weak bisimulation, weak simulation},
}

\begin{document}
\maketitle

\begin{abstract}
In a recent paper, we have introduced two types of fuzzy simulations (forward and backward) and five types of fuzzy bisimulations (forward, backward, forward-backward, backward-forward and regular) between Kripke models for the fuzzy multimodal logics over a complete linearly ordered Heyting algebra. In this paper, for a given non-empty set $\Psi $ of modal formulae, we introduce the concept of a weak bisimulation between Kripke models. This concept can be used to express the degree of equality of fuzzy sets of formulae from $\Psi $ that are valid in two worlds $w$ and $w'$, that is, to express the degree of modal equivalence between worlds $w$ and $w'$ with respect to the formulae from $\Psi$. We prove several Hennessy-Milner type theorems. The first theorem determines that the greatest weak bisimulation for the set of plus-formulae between image-finite Kripke models coincides with the greatest forward bisimulation. The second theorem determines that the greatest weak bisimulation for the set of minus-formulae between domain-finite Kripke models coincides with the greatest backward bisimulation. Finally, the third theorem determines that the greatest weak bisimulation for the set of all modal formulae between the degree-finite Kripke models coincides with the greatest regular bisimulation.
\end{abstract}


\section{Introduction}

Fuzzy logic is a form of multi-valued logic popularized by Zadeh's work \cite{Zadeh1965}, although it had been studied before by \L ukasiewicz and Tarski \cite{Lukasiewicz1930}. This approach shifts the paradigm from the standard set of Boolean truth values to a more general lattice from which the formula can take the truth value. Fuzzy first-order logic is obtained by applying the fuzzy approach to first order logic (cf.~\cite{Hajek1998,Novak1987}). Furthermore, after some early attempts to combine fuzzy logic and modal logic (see, for example, \cite{Schotch1976}), the development of fuzzy modal logic progressed rapidly (see \cite{Bou2011, Fitting1991, Ostermann1988, Priest2008}). A special type of fuzzy modal logic should be particularly emphasized - fuzzy description logic which has flourished over the last few decades (for a detailed survey, see \cite{Borgwardt2017}).

A significant milestone in the research of modal logics, automata, labelled transition systems (LTS), etc., is the introduction of \textit{bisimulations}. This is a multifaceted concept which offers some powerful tools for defining, understanding and reasoning about objects and structures that are common in mathematics and computer science. It could be said that bisimulations were introduced at about the same time in several areas independently. For example, in concurrency theory, the origin of bisimulations can be found in the works of R.~Milner \cite{Milner1980}, M.~Hennessy and R.~Milner \cite{HennessyMilner1980, HennessyMilner1985} and D.~Park \cite{Park1981}. Also, van Benthem in \cite{Benthem1977} defined bisimulation in the model theory of modal logic under the name of $p$-\textit{relations} or \textit{zig-zag} relations. For more information on the origins of bisimulations and their applications we refer to \cite{Sangiorgi2009, Sangiorgi2011}. Most researchers who have dealt with simulations and bisimulations for various types of relational systems have considered only forward simulations and forward bisimulations. They have used the names strong simulations and strong bisimulations, or just simulations and bisimulations (cf.~\cite{Fan2015, Milner1989, Milner1999, Roggenbach2000}). The greatest bisimulation equivalence is usually called a \textit{bisimilarity}.

Kripke models in modal logic, automata and the labelled transition systems are very similar syntactically. Looking from that perspective, the evaluation of the formulae in modal logic can be viewed as automata computation or computing LTS and vice versa. Therefore, the ideas and results from one theory can be taken into consideration in the other two.

Bearing this in mind, we have recently defined two types of simulations (forward and backward) and five types of bisimulations (forward, back\-ward, forward-back\-ward, backward-forward and regular) (see \cite{Stankovic2021}) between two fuzzy Kripke models. The definitions are based on \cite{Ciric2012} where two types of simulations and four types of bisimulations for the fuzzy finite automata have been studied. The fifth type of bisimulations, called regular bisimulations, originate from the research on the fuzzy social networks \cite{Ignjatovic2015}. What is more, we have created an algorithm that tests the existence of various types of simulation or bisimulation between the given Kripke models. We have also applied bisimulations in the state reduction of the fuzzy Kripke models, while preserving their semantic properties. It turns out that defining different types of bisimulations is not in vain. Namely, when forward, backward or regular bisimulation is fuzzy quasi-order, we have constructed the corresponding fuzzy Kripke model with smaller sets of worlds which is equivalent to the original one with respect to plus-formulae, minus-formulae and all formulae.

The Hennessy\-{-}Milner property, i.e., the property when modal equivalence coincides with bisimilarity for the image-finite or modally saturated models is well-known in modal logic. The question whether the Hennessy-Milner property holds for fuzzy modal logic is not the easy one, and remains mostly unexplored, although there are several papers on the subject. It is significant to mention the work of Fan (cf.~\cite{Fan2015}) who defined a fuzzy bisimulation for standard G\"{o}del modal logic and its extension with converse modalities and proved the Hennessy-Milner type theorem for these logics. Also, Eleftheriou et al.~\cite{Eleftheriou2012} examined the notion of bisimulations for many-valued modal languages over Heyting algebras. They defined notions like $t$-\textit{invariance}, $t$-\textit{bisimilarity} and also the notion of \textit{weak bisimulation}. In addition, they showed that for the image-finite models, $t$-invariance implies $t$-bisimilarity. We also need to mention other papers dealing with this subject such as \cite{Marti2018, Marti2014} where the Hennessy-Milner property was investigated for many-valued logic with a crisp accessibility relation; \cite{Bilkova2016} where the Hennessy-Milner property was investigated via coalgebraic methods; \cite{Diaconescu2020} where the Hennessy-Milner property was investigated for many-valued modal logic with a many-valued accessibility relation; as well as the research in fuzzy description logic \cite{Nguyen2019, Nguyen2021a, Nguyen2020}, etc.

Our primary goal is to prove several Hennessy-Milner type theorems for fuzzy multimodal logics over linearly ordered Heyting algebras. We introduce the concept of a weak bisimulation for a given non-empty set $\Psi $ of modal formulae, which can be used to express the degree of equality of fuzzy sets of formulae from $\Psi $ that are valid in two worlds $w$ and $w'$. In other words, they could be used to express the degree of modal equivalence between worlds $w$ and $w'$ with respect to the formulae from $\Psi$. We show that the greatest weak bisimulation for the set of plus-formulae between the image-finite Kripke models coincides with the greatest forward bisimulation. Furthermore, we show that the greatest weak bisimulation for the set of minus-formulae between domain-finite Kripke models coincides with the greatest backward bisimulation and that the greatest weak bisimulation for the set of all modal formulae between degree-finite Kripke models coincides with the greatest regular bisimulation. This means that, in cases such as these, the degrees of modal equivalences for plus-formulae, minus-formulae and all modal formulae can be expressed using the greatest forward, backward and regular bisimulations. 

These results are important for several reasons. The modal equivalence test for a given set of formulae comes down to computing the greatest weak bisimula\-tion corresponding to that set of formulae, which is generally a com\-pu\-ta\-tionally hard problem. Our results reduce such problems to the problems of computing the greatest forward, backward and regular bisimulations, for which efficient algorithms have been developed in \cite{Stankovic2021}.

The results obtained from this research may have various potential applications. Our modal language syntax is inter-translatable with the syntax of the fuzzy description logics (cf.~\cite{Bobillo2009, Bobillo2012, Hajek2005}), fuzzy temporal logic \cite{Conradie2020} and social network analysis \cite{Fan2013, Fan2014, Ignjatovic2015}. In fact, a weighted social network can easily be transformed into a Kripke model (see section 3 from \cite{Fan2014}). Therefore, the results presented in this paper can be applied in social network analysis, especially when regular equivalence computation is required.

The paper is organized into eight sections. The Introduction is followed by Section \ref{Preliminaries} which include some basic relevant definitions and notations for Heyting algebras, fuzzy sets and fuzzy relations. Section \ref{Section - Fuzzy Multimodal Logics} reviews the syntax and semantics for the fuzzy multimodal logics over a Heyting algebra. Section \ref{Section - Simulations and bisimulations} reviews two types of simulations and five types of bisimulations between two fuzzy Kripke models and defines the notions of weak simulations and weak bisimulations for some sets of formulae. The main results of the paper, the Hennessy-Milner type theorems are presented and proved in Section \ref{Section - Hennessy-Milner type theorems for fuzzy multimodal logics}. Section \ref{Section - Hennessy-Milner type theorems for Propositional Modal Logics} reformulates the theorems from the previous section for the special case of propositional modal logics. Then, Section \ref{Section - Computational examples} presents some computational examples which demonstrate applications of the results from Sections \ref{Section - Hennessy-Milner type theorems for fuzzy multimodal logics}, \ref{Section - Hennessy-Milner type theorems for Propositional Modal Logics} and \cite{Stankovic2021}. Finally, Section \ref{Concluding Remarks} contains some concluding remarks.

\section{Preliminaries}\label{Preliminaries}

Firstly, we will briefly list all the necessary terms and definitions from \cite{Stankovic2021}.
\begin{definition}\label{Definition of Heyting algebra}
An algebra $\mathscr{H}=(H, \wedge, \vee, \rightarrow, 0, 1)$ with three binary and two nullary operations is a \textit{Heyting algebra} if it satisfies:
\begin{itemize}
\item[{\rm (H1)}] $(H, \wedge, \vee)$ is a distributive lattice;
\item[{\rm (H2)}] $x\wedge 0 = 0$,\quad $x\vee 1 = 1$;
\item[{\rm (H3)}] $x\rightarrow x =1$;
\item[{\rm (H4)}] $(x\rightarrow y)\wedge y = y$,\quad $x\wedge (x\rightarrow y)=x\wedge y$;
\item[{\rm (H5)}] $x\rightarrow (y\wedge z)=(x\rightarrow y)\wedge (x\rightarrow z)$,\quad $(x\vee y)\rightarrow z=(x\rightarrow z)\wedge (y\rightarrow z)$.
\end{itemize}
\end{definition}

A binary operation $\rightarrow$ is called \textit{relative pseu\-do\-com\-ple\-men\-ta\-tion}, or \textit{re\-si\-du\-um}, in many sources. The \textit{relative pseu\-do\-com\-ple\-ment $x\rightarrow z$ of $x$ with respect to $z$} can be characterized as follows:
\begin{equation}
x\rightarrow z =\bigvee \bigl\{y\in H\mid x\wedge y \leqslant z\bigr\}\mbox{.}
\end{equation}
Equivalently, we say that operations $\wedge$ and $\rightarrow$ form an \textit{adjoint pair}, i.e., they satisfy the \textit{adjunction property} or \textit{resi\-duation property}: for all $x,y, z \in H$,
\begin{equation}\label{adjunction property in Heyting algebra}
x\wedge y \leqslant z \qquad \Leftrightarrow \qquad x\leqslant y\rightarrow z\mbox{.}
\end{equation}
If, in addition, $(H, \wedge, \vee, 0, 1)$ is a complete lattice, then $\mathscr{H}$ is called a \textit{complete Heyting algebra}. In the rest of the paper, unless otherwise stated, $\mathscr{H} = (H,\wedge,\vee,\rightarrow, 0, 1)$ stands for a complete Heyting algebra.

The operation $\leftrightarrow$ defined by
\begin{equation}\label{definition of biimplication}
x\leftrightarrow y = (x\rightarrow y) \wedge (y\rightarrow x)\mbox{,}
\end{equation}
called \textit{bi-implication}, is used for modeling the equivalence of truth values.

Now we can define \textit{fuzzy subset}, \textit{fuzzy relations} and other terms with their properties over $\mathscr{H}$. Also, it is generally known that the Heyting algebra $\mathscr{H}=(H, \wedge, \vee, \rightarrow, 0, 1)$ can be defined as a commutative residuated lattice $\mathscr{H}=(H, \wedge, \vee, \otimes, \rightarrow, 0, 1)$ in which operation $\otimes$ coincides with $\wedge$. Therefore, the following definitions and terminology are based on \cite{Belohvalek2002, Belohvalek2005} where they are given for a residuated lattice. 

\begin{definition}\label{Definition - fuzzy subset}
A \textit{fuzzy subset} of a set $A$ \textit{over} $\mathscr{H}$, or simply a \textit{fuzzy subset} of $A$ is a function from $A$ to $H$. Ordinary crisp subsets of $A$ are considered fuzzy subsets of $A$ taking membership values in the set $\{0,1\}\subseteq H$.
\end{definition}

Let $f$ and $g$ be two fuzzy subsets of $A$. The \textit{equality} of $f$ and $g$ is defined as the usual equality of functions, i.e., $f=g$ if and only if $f(x)=g(x)$, for every $x\in A$. The \textit{inclusion} $f\leqslant g$ is also defined as usual: $f\leqslant g$ if and only if $f(x) \leqslant g(x)$, for every $x\in A$. With this partial order, the set $\mathscr{F}(A)$ of all fuzzy subsets of $A$ forms a complete Heyting algebra, in which the meet (intersection) $\bigwedge_{i \in I}f_i$ and the join (union) $\bigvee_{i\in I} f_i$ of an arbitrary family $\{f_i\}_{i \in I}$ of fuzzy subsets of $A$ are functions from $A$ to $H$ defined by
\begin{equation*}
\left(\bigwedge_{i \in I} f_i\right)(x) = \bigwedge_{i \in I} f_i (x)\mbox{,} \qquad \left(\bigvee_{i \in I} f_i\right)(x) = \bigvee_{i \in I} f_i (x)\mbox{.}
\end{equation*}
Note that the equality, inclusion, meet and join of fuzzy sets are all defined pointwise. We can define the \textit{product} $f \wedge g$ to be the same as the binary meet: $f\wedge g (x) = f(x) \wedge g(x)$, for every $x\in A$, due to the relationship between a Heyting algebra and a residuated lattice.

\begin{definition}
Let $A$ and $B$ be non-empty sets. A \textit{fuzzy relation between sets} $A$ and $B$ (in this order) is a function from $A\times B$ to $H$, i.e., a fuzzy subset of $A\times B$, and the equality, inclusion (ordering), joins and meets of fuzzy relations are defined as in the case of fuzzy sets.
\end{definition}
In particular, a \textit{fuzzy relation on a set} $A$ is a function from $A\times A$ to $H$, i.e., a fuzzy subset of $A\times A$. The set of all fuzzy relations from $A$ to $B$ will be denoted by $\mathscr{R}(A, B)$, and the set of all fuzzy relations on a set $A$ will be denoted by $\mathscr{R}(A)$. The \textit{inverse} of a fuzzy relation $\varphi \in \mathscr{R}(A, B)$ is a fuzzy relation $\varphi^{-1} \in \mathscr{R}(B, A)$ defined by $\varphi^{-1}(b,a)=\varphi(a,b)$, for all $a\in A$ and $b\in B$. A \textit{crisp relation} is a fuzzy relation which takes values only in the set $\{0,1\}$, and if $\varphi$ is a crisp relation of $A$ to $B$, then the expressions ``$\varphi(a,b)=1$'' and ``$(a,b)\in \varphi$'' will have the same meaning.

\begin{definition}
For non-empty sets $A$, $B$ and $C$, and fuzzy relations $\varphi\in \mathscr{R}(A, B)$ and $\psi\in \mathscr{R}(B,C)$, their \textit{composition} $\varphi \circ \psi$ is a fuzzy relation from $\mathscr{R}(A,C)$ defined by
\begin{equation}
(\varphi\circ \psi)(a,c)=\bigvee_{b\in B}\varphi(a,b)\wedge \psi(b,c)\mbox{,}\label{composition of fuzzy relations}
\end{equation}
for all $a \in A$ and $c \in C$.
\end{definition}
If $\varphi$ and $\psi$ are crisp relations, then $\varphi \circ \psi$ is an ordinary composition of relations, i.e.,
\begin{equation*}
\varphi \circ \psi = \{(a,c)\in A\times C \mid (\exists b \in B)(a,b)\in \varphi\; \&\;(b,c)\in \psi\}\mbox{,}
\end{equation*}
and if $\varphi$ and $\psi$ are functions, then $\varphi\circ \psi$ is the ordinary composition of functions, i.e., $(\varphi \circ \psi)(a) = \psi(\varphi(a))$, for every $a\in A$.
\begin{definition}
Let $f \in \mathscr{F}(A)$, $\varphi \in \mathscr{R}(A, B)$ and $g \in \mathscr{F}(B)$. The compositions $f \circ \varphi$ and $\varphi \circ g$ are fuzzy subsets of $B$ and $A$, respectively, which are defined by
\begin{equation}\label{composition of fuzzy set and fuzzy relation}
(f\circ \varphi)(b)=\bigvee_{a\in A}f(a)\wedge \varphi(a,b)\mbox{,}\qquad
(\varphi\circ g)(a)=\bigvee_{b\in B} \varphi(a,b)\wedge g(b)\mbox{,}
\end{equation}
for every $a \in A$ and $b \in B$.
\end{definition}
In particular, if $f$ and $g$ are crisp sets and $\varphi$ is a crisp relation, then
\begin{equation*}
f\circ \varphi = \{b\in B\mid (\exists a \in f)(a,b)\in \varphi\}\mbox{,}\qquad
\varphi \circ g = \{a\in A \mid (\exists b\in g)(a,b)\in \varphi\}\mbox{.}
\end{equation*}

\begin{definition}
Let $f,g\in \mathscr{F}(A)$. The composition $f\circ g$ is an element of a fuzzy set $A$, defined by
\begin{equation}
f\circ g = \bigvee_{a\in A} f(a)\wedge g(a)\mbox{.}
\end{equation}
\end{definition}

We note that if $A,B$ and $C$ are finite sets of cardinality $|A|=k$, $|B|=m$ and $|C|=n$, then $\varphi \in \mathscr{R}(A, B)$ and $\psi \in \mathscr{R}(B, C)$ can be treated as $k\times m$ and $m\times n$ fuzzy matrices over $\mathscr{H}$, and $\varphi \circ \psi$ is the matrix product. Analogously, for $f\in \mathscr{F}(A)$ and $g\in \mathscr{F}(B)$ we can treat $f\circ \varphi$ as the product of a $1\times k$ matrix $f$ and a $k\times m$ matrix $\varphi$, and $\varphi\circ g$ as the product of a $k\times m$ matrix $R$ and an $m\times 1$ matrix $g^t$ (the transpose of $g$).

The following lemmas give the basic properties of the composition of fuzzy relations and fuzzy subsets.
\begin{lemma}
Let $A,B,C$ and $D$ be non-empty sets. Then we have:
\begin{itemize}
\item [{\rm a)}] For $\varphi_1 \in \mathscr{R}(A,B)$, $\varphi_2 \in \mathscr{R}(B, C)$ and $\varphi_3\in \mathscr{R}(C, D)$ we have
\begin{equation*}\label{fuzzy relations compositions are associative}
(\varphi_1 \circ \varphi_2)\circ \varphi_3 = \varphi_1 \circ (\varphi_2 \circ \varphi_3)\mbox{.}
\end{equation*}
\item [{\rm b)}] For $\varphi_0 \in \mathscr{R}(A, B)$, $\varphi_1, \varphi_2 \in \mathscr{R}(B, C)$ and $\varphi_3 \in \mathscr{R}(C, D)$ we have that $\varphi_1 \leqslant \varphi_2$ implies $\varphi_1 ^{-1}\leqslant \varphi_2 ^{-1}$, $\varphi_0 \circ \varphi_1 \leqslant \varphi_0 \circ \varphi_2$ and $\varphi_1 \circ \varphi_3 \leqslant \varphi_2 \circ \varphi_3$.
\item [{\rm c)}] For a $\varphi \in \mathscr{R}(A, B)$, $\psi \in \mathscr{R}(B, C)$, $f\in \mathscr{F}(A)$, $g\in \mathscr{F}(B)$ and $h\in \mathscr{F}(C)$ the following holds:
\[
(f\circ \varphi)\circ \psi = f\circ (\phi\circ \psi)\mbox{,}\qquad
(f\circ \varphi)\circ g = f \circ (\varphi \circ g)\mbox{,}\qquad
(\varphi\circ \psi)\circ h = \varphi \circ (\psi \circ h)\mbox{.}
\]
\end{itemize}
\end{lemma}
 
Consequently, in the previous lemma, the parentheses in a) and c) can be omitted.

\begin{lemma}
For all $\varphi, \varphi_i \in\mathscr{R}(A, B)(i\in I)$ and $\psi, \psi_i \in \mathscr{R}(B, C)\allowbreak(i\in I)$ we have that
\begin{equation}
(\varphi\circ \psi)^{-1} = \psi^{-1}\circ \varphi^{-1}\mbox{,}
\end{equation}
\begin{equation}
\varphi \circ \left(\bigvee_{i\in I}\psi_i\right) = \bigvee_{i\in I}(\varphi \circ \psi_i)\mbox{,}\qquad \left(\bigvee_{i\in I}\varphi_i\right) \circ \psi = \bigvee_{i\in I}(\varphi_i\circ \psi)\mbox{,}
\end{equation}
\begin{equation}
\left(\bigvee_{i\in I}\varphi_i\right)^{-1}= \bigvee_{i\in I}\varphi_i ^{-1}\mbox{.}
\end{equation}
\end{lemma}
\begin{definition}\label{Definition - image, domain and degree finite relation}
Let $A$ and $B$ be fuzzy sets. A fuzzy relation $\varphi\in \mathscr{R}(A, B)$ is called \textit{image-finite} if for every $a\in A$ the set $\{b\in B\mid \varphi(a,b)>0\}$ is finite, it is called \textit{domain-finite} if for every $b\in B$ the set $\{a\in A\mid \varphi(a,b)>0\}$ is finite, and it is called \textit{degree-finite} if it is both image-finite and domain-finite.
\end{definition}

\section{Fuzzy multimodal logics}\label{Section - Fuzzy Multimodal Logics}
In \cite{Stankovic2021} we introduced a fuzzy multimodal logic over a complete Heyting algebra, and here we will give a brief overview of the relevant definitions.

\begin{definition}\label{Definition - Language over H}
Let $\mathscr{H}=(H, \wedge, \vee, \rightarrow, 0, 1)$ be a complete Heyting algebra and write $\overline{H}=\{\overline{t}\mid t\in H\}$ for the elements of $\mathscr{H}$ viewed as constants. Let $I$ be some index set. Define the language $\Phi_{I,\mathscr{H}}$ via the grammar
\[
A::=\overline{t}\mid p\mid A\wedge A\mid A\rightarrow A\mid \square_i A\mid \lozenge_i A\mid \squaremi A\mid \lozengemi A
\]
where $\overline{t}\in \overline{H}$, $i\in I$ and $p$ ranges over some set $PV$ of proposition letters.
\end{definition}

The following well-known abbreviations will be used:
\begin{flalign*}
&\neg A \equiv A\rightarrow \overline{0} \mbox{ (\textit{negation})},&&\\
&A \leftrightarrow B \equiv (A\rightarrow B)\wedge (B\rightarrow A) \mbox{ (\textit{equivalence}),}&&\\
& A\vee B \equiv ((A\rightarrow B)\rightarrow B)\wedge ((B\rightarrow A)\rightarrow A) \mbox{ (\textit{disjunction}).}&&
\end{flalign*}
Recall that $0$ is the least element in $\mathscr{H}$ and $\overline{0}$ is the corresponding truth constant. Also, $\overline{0}\rightarrow \overline{0}$ gives $\overline{1}$.

The set of all formulae over the alphabet $\mathscr{H}(\{\square_i, \lozenge_i\}_{i\in I})$, i.e., the set of those formulae from $\PhiIH$ that do not contain any of the modal operators $\squaremi$ and $\lozengemi$, $i\in I$, will be denoted by $\PhiIHplus$. Similarly, the set of all formulae over the alphabet $\mathscr{H}(\{{\squaremi},{\lozengemi}\}_{i\in I})$, i.e., the set of those formulae from $\PhiIH$ that do not contain any of the modal operators $\square_i$ and $\lozenge_i$, $i\in I$, will be denoted by $\PhiIHminus$. To simplify, the formulae from $\PhiIHplus$ will be called \textit{plus-formulae}, and the formulae from $\PhiIHminus$ will be called \textit{minus-formulae}.

\begin{definition}\label{Definition - fuzzy Kripke frame}
A \textit{fuzzy Kripke frame} is a structure $\mathfrak{F} = (W,\{R_i\}_{i\in I})$ where $W$ is a non-empty set of possible \textit{worlds} (or \textit{states} or \textit{points}) and $R_i\in \mathscr{F}(W\times W)$ is a binary fuzzy relation on $W$, for every $i$ from a finite index set $I$, called the \textit{accessibility fuzzy relation} of the frame.
\end{definition}

\begin{definition}\label{Definition - fuzzy Kripke model}
A \textit{fuzzy Kripke model for $\PhiIH$\/} is a structure $\mathfrak{M}=(W,\allowbreak \{R_i\}_{i\in I},\allowbreak V)$ such that $(W,\{R_i\}_{i\in I})$ is a fuzzy Kripke frame and $V: W\times (PV \cup \overline{H}) \rightarrow H$ is a truth assignment function, called the \textit{evaluation of the model}, which assigns an $H$-truth value to propositional variables (and truth constants) in each world, such that $V(w,\overline{t})=t$, for every $w\in W$ and $t\in H$.
\end{definition}

In the case when the finite set $I$ has $n$ elements, then $\mathfrak{F}$ is called a \textit{fuzzy Kripke $n$-frame} and $\mathfrak{M}$ is called a \textit{fuzzy Kripke $n$-model}.

The truth assignment function $V$ can be inductively extended to a function $V:W\times \PhiIH \rightarrow H$ by:
\begin{flalign*}
\textrm{(V1) }&V(w, A\wedge B)=V(w, A)\wedge V(w, B);&&\\
\textrm{(V2) }&V(w, A\rightarrow B)=V(w, A)\rightarrow V(w, B);&&\\
\textrm{(V3) }&V(w, \square_{i} A)=\bigwedge_{u\in W} R_{i}(w,u)\rightarrow V(u, A)\mbox{, for every }i\in I;&&\\
\textrm{(V4) }&V(w, \lozenge_{i} A)=\bigvee_{u\in W} R_{i}(w, u)\wedge V(u, A)\mbox{, for every }i\in I;&&\\
\textrm{(V5) }&V(w, \squaremi A)=\bigwedge_{u\in W} R_{i}(u,w)\rightarrow V(u, A)\mbox{, for every }i\in I;&&\\
\textrm{(V6) }&V(w, \lozengemi A)=\bigvee_{u\in W} R_{i}(u, w)\wedge V(u, A)\mbox{, for every }i\in I.&&
\end{flalign*}
For each world $w\in W$, the truth assignment $V$ determines a function $V_w:\PhiIH\to H$ given by $V_w(A)=V(w,A)$, for every $A\in \PhiIH$, and similarly, for each $A\in \PhiIH $, the truth assignment $V$ determines a function $V_A:W\to H$ given by $V_A(w)=V(w,A)$, for every $w\in W$.

As a rule, we denote the models with $\mathfrak{M}$, $\mathfrak{M}'$, $\mathfrak{N}$, $\mathfrak{N}'$ etc., not emphasizing the alphabet $\mathscr{H}(\{\square_i,\lozenge_i,\squaremi,\lozengemi\}_{i\in I})$ specifically, except when necessary.
For a fuzzy Kripke model $\mathfrak{M}=(W, \{R_{i}\}_{i\in I}, V)$, its \textit{reverse fuzzy Kripke model} is the fuzzy Kripke model $\mathfrak{M}^{-1}=(W, \{R^{-1}_{i}\}_{i\in I}, V)$.

The following Definition is based on Definition \ref{Definition - image, domain and degree finite relation} which defines image-finite, domain-finite and degree-finite relations.
\begin{definition}
A fuzzy Kripke model $\mathfrak{M}=(W, \{R_{i}\}_{i\in I}, V)$ is called \textit{image-finite\/} if the relation $R_i$ is image-finite, for every $i\in I$, it is called \textit{domain-finite\/} if the relation $R_i$ is domain-finite, for every $i\in I$, and it is called \textit{degree-finite\/} if the relation $R_i$ is degree-finite, for every $i\in I$.
\end{definition}

\begin{definition}\label{Definition - formulae-equivalent}
Let $\mathfrak{M}=(W, \{R_i\}_{i\in I}, V)$ and $\mathfrak{M}'=(W', \{R_i'\}_{i\in I}, V')$ be two fuzzy Kripke models, and let $\Phi \subseteq \PhiIH$ be some set of formulae. The worlds $w\in W$ and $w'\in W'$ are said to be \textit{$\Phi$-equivalent} if $V(w,A)=V'(w',A)$, for all $A\in \Phi$. Moreover, $\mathfrak{M}$ and $\mathfrak{M}'$ are said to be \textit{$\Phi$-equivalent fuzzy Kripke models} if each $w\in W$ is $\Phi$-equivalent to some $w'\in W'$, and vice versa, if each $w'\in W'$ is $\Phi $-equivalent to some $w\in W$.
\end{definition}

Many authors use the term \textit{modal equivalence} for the relation between two worlds defined as follows: two worlds $w\in W$ and $w'\in W$ are \textit{modally equivalent} if $V(w,A)=V'(w',A)$, where $A$ is from the set of all formulae (cf.~\cite{Blackburn2001, Diaconescu2020}). Therefore, Definition \ref{Definition - formulae-equivalent} is more general since the notion of formulae equivalence can be defined for some set of formulae.

\section{Simulations and bisimulations}\label{Section - Simulations and bisimulations}
\subsection{Simulations and bisimulations}\label{subs:sim.bisim}

In a fuzzy modal logic, fuzzy simulation relates a fuzzy Kripke model to an \textit{abstraction} of the model where the abstraction of the model might have a smaller set of worlds. Hence, the fuzzy Kripke model is related to his abstraction in such a way that every local property and transition patterns of worlds, relevant to the simulation requirement, are preserved. Therefore, fuzzy bisimulations guarantee that two fuzzy Kripke models have the same local properties and transition patterns.

Two types of simulations and five types of bisimulations between two fuzzy Kripke models are defined in \cite{Stankovic2021}. The definitions are based on \cite{Ciric2012} in which two types of simulations and four types of bisimulations for fuzzy finite automata have been studied. The fifth type of bisimulations, called regular bisimulations, originates from the research on fuzzy social networks \cite{Ignjatovic2015}.

In the crisp case, a bisimulation preserves logical equivalence, i.e., modal formulae are invariant under bisimulation (see Theorem 2.20 from \cite{Blackburn2001}). We consider three defined sets of modal formulae (plus-formulae, minus-formulae and all formulae), so we need one type of bisimulation for each set of formulae. Also, there are two mixed types of bisimulations, which makes for a total of five types of bisimulations that we define.

Now, we will briefly outline the definitions.

\begin{definition}\label{Definition - forward and backward simulation}
Let $\mathfrak{M}=(W, \{R_{i}\}_{i\in I}, V)$ and $\mathfrak{M}'=(W', \{R_{i}'\}_{i\in I}, V')$ be two fuzzy Kripke models and let $\varphi\in \mathscr{R}(W, W')$ be a non-empty fuzzy relation. If $\varphi $ satisfies
\begin{align}
& V_p \leqslant V_p '\circ \varphi^{-1}\mbox{,}\qquad\mbox{ for every } p\in PV\mbox{,}\label{forward simulation -1} \tag{\textit{fs-1}}\\
& \varphi^{-1} \circ R_{i} \leqslant R_{i}' \circ \varphi^{-1}\mbox{,}\qquad\mbox{for every } i\in I\mbox{,} \label{forward simulation -2} \tag{\textit{fs-2}} \\
& \varphi^{-1}\circ V_p \leqslant V_p'\mbox{,}\qquad\mbox{for every } p\in PV\mbox{,} \label{forward simulation -3} \tag{\textit{fs-3}}
\end{align}
then it is called a \textit{forward simulation} between $\mathfrak{M}$ and $\mathfrak{M}'$, and if it satisfies only \eqref{forward simulation -2} and \eqref{forward simulation -3}, then it is called a \textit{forward pre\-simulation} between $\mathfrak{M}$ and $\mathfrak{M}'$.

On the other hand, if $\varphi $ satisfies
\begin{align}
& V_p \leqslant \varphi\circ V_p ' \mbox{,}\qquad \mbox{for every } p\in PV\mbox{,}\label{backward simulation -1} \tag{\textit{bs-1}}\\
& R_{i} \circ \varphi \leqslant \varphi \circ R_{i}'\mbox{,}\qquad\mbox{for every } i\in I\mbox{,} \label{backward simulation -2} \tag{\textit{bs-2}}\\
& V_p \circ \varphi \leqslant V_p'\mbox{,}\qquad \mbox{for every } p\in PV\mbox{,} \label{backward simulation -3} \tag{\textit{bs-3}}
\end{align}
then it is called a \textit{backward simulation} between $\mathfrak{M}$ and $\mathfrak{M}'$, and if it satisfies only \eqref{backward simulation -3} and \eqref{backward simulation -2}, it is called a \textit{backward presimulation} between $\mathfrak{M}$ and $\mathfrak{M}'$.
\end{definition}

The meaning of forward and backward simulations can best be explained in the case when $\mathfrak{M}$ and $\mathfrak{M}'$ are crisp (Boolean-valued) Kripke models and $\varphi $ is an ordinary crisp (Boolean-valued) binary relation. The condition \eqref{forward simulation -1} means that if the valuation $V$ assigns the value ``true'' to the propositional variable $p$ in some world $w$, then the valuation $V'$ assigns to this variable the value ``true'' in some world $w'$ which simulates $w$. On the other hand, the condition \eqref{forward simulation -3} means that if $w'$ simulates $w$ and the valuation $V$ assigns the value ``true'' to the propositional variable $p$ in the world $w$, then the valuation $V'$ also assigns the value ``true'' to this variable in the world $w'$. The conditions \eqref{forward simulation -2} and \eqref{backward simulation -2} can be explained as follows: \eqref{forward simulation -2} means that if $u'$ simulates $u$ and $v$ is accessible from $u$, then there is $v'$ accessible from $u'$ which simulates $v$, and \eqref{backward simulation -2} means that if $u$ is accessible from $v$ and $u'$ simulates $u$, then $u'$ is accessible from some $v'$ which simulates $v$. This is explained in Figure \ref{figure}. In both cases, accessibility is considered with respect to $R_i$, for each $i\in I$.
\begin{figure}[!ht]
\centering
\begin{tikzpicture}
\tikzstyle{arrow} = [thick, ->,>=latex]
\draw [arrow](0,3) node[circle,fill,inner sep=1pt,label=above left:$u$](u){} --node[anchor=south] {$\varphi$} (3,3)node[circle,fill, inner sep=1pt,label=above right:$u'$](u'){};

\draw [arrow](0,3) (u){} --node[anchor=east] {$R_i$} (0,0)node[circle,fill,inner sep=1pt,label=below left:$v$](v){};

\draw [arrow](6,3) node[circle,fill,inner sep=1pt,label=above left:$u$](u){} --node[anchor=south] {$\varphi$} (9,3)node[circle,fill,inner sep=1pt,label=above right:$u'$](u'){};

\draw [arrow](6,0) node[circle,fill,inner sep=1pt,label=below left:$v$](v){} --node[anchor=east] {$R_i$} (6,3)node(u){};

\tikzstyle{arrow} = [dashed, ->,>=latex]

\draw [arrow](3,3) node(u'){} --node[anchor=west] {$R_i '$} (3,0)node[circle,fill,inner sep=1pt,label=below right:$v'$](v'){};

\draw [arrow](0,0) node(v){} --node[anchor=north] {$\varphi$} (3,0)node(v'){};

\draw [arrow](9,0) node[circle,fill,inner sep=1pt,label=below right:$v'$](v'){} --node[anchor=west] {$R_i '$} (9,3)node(u'){};

\draw [arrow](6,0) node(v){} --node[anchor=north] {$\varphi$} (9,0)node(v'){};
\end{tikzpicture}
\caption{Forward simulation (the condition \eqref{forward simulation -2}, on the left) and backward simulation (the condition \eqref{backward simulation -2}, on the right).}
\label{figure}
\end{figure}
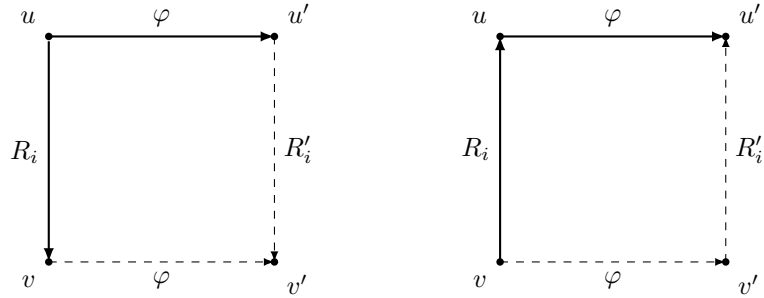
Now, we can define five types of bisimulations by combining notions of forward and backward simulations.

\begin{definition}\label{Definition - forward bisimulation}
Let $\mathfrak{M}=(W, \{R_{i}\}_{i\in I}, V)$ and $\mathfrak{M}'=(W', \{R_{i}'\}_{i\in I}, V')$ be two fuzzy Kripke models and let $\varphi\in \mathscr{R}(W, W')$ be a non-empty fuzzy relation. If both $\varphi$ and $\varphi^{-1}$ are forward simulations, i.e., if
\begin{align}
& V_p \leqslant V_p '\circ \varphi^{-1}\mbox{,}\quad V_p ' \leqslant V_p \circ \varphi\mbox{,}\qquad \mbox{ for every } p\in PV\mbox{,}\label{forward bisimulation -1}\tag{\textit{fb-1}}\\
& \varphi^{-1} \circ R_{i} \leqslant R_{i}' \circ \varphi^{-1}\mbox{,}\quad \varphi \circ R_i' \leqslant R_i \circ \varphi \mbox{,}\qquad \mbox{for every } i\in I\mbox{,}\label{forward bisimulation -2}\tag{\textit{fb-2}}\\
& \varphi^{-1}\circ V_p \leqslant V_p'\mbox{,} \quad \varphi \circ V_p' \leqslant V_p\mbox{,}\qquad \mbox{for every } p\in PV\mbox{.}\label{forward bisimulation -3}\tag{\textit{fb-3}}
\end{align}
then we call $\varphi$ a \textit{forward bisimulation} between $\mathfrak{M}$ and $\mathfrak{M}'$, and if $\varphi$ satisfies only \eqref{forward bisimulation -2} and \eqref{forward bisimulation -3}, then we call it a \textit{forward prebisimulation} between $\mathfrak{M}$ and $\mathfrak{M}'$.
\end{definition}
Now, we will define a backward (pre)bisimulation.
\begin{definition}
Let $\mathfrak{M}=(W, \{R_{i}\}_{i\in I}, V)$ and $\mathfrak{M}'=(W', \{R_{i}'\}_{i\in I}, V')$ be two fuzzy Kripke models and let $\varphi\in \mathscr{R}(W, W')$ be a non-empty fuzzy relation. If both $\varphi$ and $\varphi^{-1}$ are backward simulations, i.e. if
\begin{align}
& V_p \leqslant \varphi\circ V_p ' \mbox{,} \quad V_p ' \leqslant \varphi^{-1}\circ V_p \mbox{,}\qquad \mbox{for every } p\in PV\mbox{,}\label{backward bisimulation -1} \tag{\textit{bb-1}}\\
& R_{i} \circ \varphi \leqslant \varphi \circ R_{i}'\mbox{,}\quad R_{i}' \circ \varphi^{-1} \leqslant \varphi^{-1}\circ R_{i} \mbox{,}\qquad \mbox{for every } i\in I\mbox{,} \label{backward bisimulation -2} \tag{\textit{bb-2}}\\
& V_p \circ \varphi \leqslant V_p'\mbox{,}\quad V_p' \circ \varphi^{-1} \leqslant V_p\mbox{,}\qquad \mbox{for every } p\in PV\mbox{.} \label{backward bisimulation -3} \tag{\textit{bb-3}}
\end{align}
then we call $\varphi $ a \textit{backward bisimulation} between $\mathfrak{M}$ and $\mathfrak{M}'$, and if $\varphi$ satisfies only \eqref{backward bisimulation -2} and \eqref{backward bisimulation -3}, then we call it a \textit{backward prebisimulation} between $\mathfrak{M}$ and $\mathfrak{M}'$.
\end{definition}

We also define two ``mixed'' types of (pre)bisimulations. Namely, if $\varphi$ is a forward simulation and $\varphi^{-1}$ is a backward simulation, then we say that $\varphi$ is a \textit{forward-backward bisimulation} between $\mathfrak{M}$ and $\mathfrak{M}'$, and if only (\textit{fbb-2}) and (\textit{fbb-3}) hold, we say that $\varphi$ is a \textit{forward-backward prebisimulation} between $\mathfrak{M}$ and $\mathfrak{M}'$.

Similarly, if $\varphi$ is a backward simulation and $\varphi^{-1}$ is a forward simulation, then we say that $\varphi$ is a \textit{backward-forward bisimulation} between $\mathfrak{M}$ and $\mathfrak{M}'$, and if only (\textit{bfb-2}) and (\textit{bfb-2}) hold, then we say that $\varphi$ is a \textit{backward-forward prebisimulation} between $\mathfrak{M}$ and $\mathfrak{M}'$.

These two bisimulations arise from analogous definitions for fuzzy automata (cf.~\cite{Ciric2012}). They can be important in structures with no order, so only equality should be used in the definitions of bisimulations (cf.~\cite{Buchholz2008}). However, the question which fragment of the formulae these bisimulations preserve for the Kripke models remains open and needs to be further explored. 

\begin{definition}
Let $\mathfrak{M}=(W, \{R_{i}\}_{i\in I}, V)$ and $\mathfrak{M}'=(W', \{R_{i}'\}_{i\in I}, V')$ be two fuzzy Kripke models and let $\varphi\in \mathscr{R}(W, W')$ be a non-empty fuzzy relation. If $\varphi$ is both a forward and a backward bisimulation, i.e., if
\begin{align}
&V_p \leqslant V_p' \circ \varphi^{-1}\mbox{,}\quad V_p' \leqslant V_p\circ \varphi \mbox{,}\quad V_p\leqslant \varphi \circ V_p'\mbox{,}\quad V_p' \leqslant \varphi^{-1}\circ V_p \mbox{,}\quad \mbox{for every } p\in PV\mbox{,}\label{regular bisimulation -1} \tag{\textit{rb-1}}\\
& \varphi^{-1} \circ R_{i} = R_{i}' \circ \varphi^{-1}\mbox{,} \quad\varphi \circ R_{i}' = R_{i}\circ \varphi \mbox{,}\qquad \mbox{for every } i\in I\mbox{,}\label{regular bisimulation -2} \tag{\textit{rb-2}} \\
&\varphi^{-1} \circ V_p \leqslant V_p'\mbox{,}\quad \varphi \circ V_p ' \leqslant V_p\mbox{,}\quad V_p\circ\varphi \leqslant V_p' \mbox{,}\quad V_p'\circ \varphi^{-1}\leqslant V_p \mbox{,}\quad \mbox{for every } p\in PV\mbox{,} \label{regular bisimulation -3} \tag{\textit{rb-3}}
\end{align}
then we call $\varphi$ a \textit{regular bisimulation} between $\mathfrak{M}$ and $\mathfrak{M}'$, and if $\varphi$ satisfies only \eqref{regular bisimulation -2} and \eqref{regular bisimulation -3}, then we call it a \textit{regular prebisimulation} between $\mathfrak{M}$ and $\mathfrak{M}'$.
\end{definition}

For any $\theta\in \{fs, bs, fb, bb, fbb, bfb,rb\}$, a fuzzy relation which satisfies ($\theta$-1), ($\theta$-2) and ($\theta$-3) will be called a \textit{simulation/bisimulation of type $\theta$} or a \textit{$\theta$-simulation/bisimulation\/} between $\mathfrak{M}$ and $\mathfrak{M}'$, and a fuzzy relation satisfying ($\theta$-2) and ($\theta$-3) will be called a \textit{presimulation/prebisimulation of type $\theta$} or a \textit{$\theta$-presimulation/prebisimulation\/} between $\mathfrak{M}$ and $\mathfrak{M}'$. In addition, if $\mathfrak{M}$ and $\mathfrak{M}'$ are the same fuzzy Kripke model, then we use the name \textit{simulation/bisimulation of type $\theta$} or \textit{$\theta$-simulation/bisimulation\/} on the fuzzy Kripke model $\mathfrak{M}$. Also, by $\varphi^\theta_{*}$ we will denote a fuzzy relation satisfying ($\theta $-2) and ($\theta $-3).

It has been noted in \cite{Stankovic2021} that every forward simulation between two fuzzy Kripke models is a backward simulation between the reverse fuzzy Kripke models. Therefore, forward and backward simulations, forward and backward bisimulations, forward-backward and backward-forward bisimulations, are mutually dual concepts while regular bisimulations are self-dual.

\begin{lemma}\label{Lemma - equality of compositions}
Let $\mathfrak{M}=(W, \{R_{i}\}_{i\in I}, V)$ and $\mathfrak{M}'=(W', \{R_{i}'\}_{i\in I}, V')$ be two fuzzy Kripke models and let $\varphi\in \mathscr{R}(W, W')$ be $\theta$-(pre)simulation/(pre)bi\-sim\-u\-lation between $\mathfrak{M}$ and $\mathfrak{M}'$. Then, the following holds:
\begin{align}
&\varphi^{-1}\circ V_p =V_p \circ \varphi\mbox{,}\quad\mbox{for every } p\in PV\mbox{,} \label{equality of composition V_p and varphi-1}\\
&\varphi\circ V_p ' = V_p ' \circ \varphi^{-1}\mbox{,}\quad\mbox{for every } p\in PV\mbox{.}\label{equality of composition V_p and varphi-2}
\end{align}
\end{lemma}
\begin{proof}
We will prove only the first case. Hence, we have
\begin{align*}
\varphi^{-1}\circ V_p (w') &= \bigvee_{w\in W} \varphi^{-1}(w', w)\wedge V_p (w) \\
&= \bigvee_{w\in W} V_p (w)\wedge \varphi(w, w')\\
&=V_p \circ \varphi (w')
\end{align*}
for every $w'\in W'$ and consequently, \eqref{equality of composition V_p and varphi-1} holds for any propositional variable $p\in PV$.
\end{proof}

Using the previous lemma, it follows that the definitions of all five types of bisimulations/prebisimulations differ only in the second conditions ($\theta$-2), for $\theta\in \{fb, bb, fbb, bfb,rb\}$. Now, conjunctions of conditions ($\theta$-1) and ($\theta$-3) in these definitions give us
\begin{equation}
V_p' = V_p \circ \varphi\mbox{,} \qquad V_p = \varphi\circ V_p ' \mbox{,}\qquad \mbox{ for every } p\in PV\mbox{.}\label{bisimulation 1-3}
\end{equation}

In a special case, when all accessibility relations and valuations are crisp, we can compare a forward bisimulation from Definition \ref{Definition - forward bisimulation} with the definition of bisimulation for a \textit{basic modal language} (i.e.~propositional modal logic with necessity and possibility operators). Our notation for the basic modal language from Section \ref{Section - Hennessy-Milner type theorems for Propositional Modal Logics} is PML$^{+}$ (PML stands for Propositional Modal Logic and should not be confused with Positive Modal Logic and Probabilistic Modal Logic). Let us recall the definition of bisimulation for PML$^{+}$ from \cite{Blackburn2001}(p.~64) with some minor notation changes for the sake of comparison.
\begin{definition}\label{Definition - bisimulation from Blackburn}
A non-empty binary relation $\varphi\subseteq W\times W'$ is called a \textit{bisimulation between} $\mathfrak{M}$ and $\mathfrak{M}'$ if the following conditions are satisfied:
\begin{itemize}
    \item [{\rm (i)}] If $(w,w')\in \varphi$ then $V(w, p)=V'(w',p)$ for every $p\in PV$.
    \item [{\rm (ii)}] If $(w,w')\in \varphi$ and $(w,v)\in R$, then there exists $v'$ (in $\mathfrak{M}'$) such that $(v,v')\in \varphi$ and $(w', v')\in R'$ (the \textit{forth condition}).
    \item [{\rm (iii)}] The converse of {\rm (ii)}: if $(w,w')\in \varphi$ and $(w', v')\in R'$, then there exists $v$ (in $\mathfrak{M}$) such that $(v,v')\in \varphi$ and $(w,v)\in R$ (the \textit{back condition}).
\end{itemize}
\end{definition}

Now, we can make a connection between our definitions of simulations and bisimulations. Firstly, let us split a base condition (i) in two parts:
\begin{itemize}
    \item [{\rm (ia)}] If $(w,w')\in \varphi$ then $V(w, p)\leqslant V'(w',p)$ for every $p\in PV$;
    \item [{\rm (ib)}] If $(w,w')\in \varphi$ then $V(w, p)\geqslant V'(w',p)$ for every $p\in PV$.
\end{itemize}
Then, from (ia), we have $\varphi^{-1}(w',w)\wedge V(w,p)\leqslant V'(w',p)$ for every $p\in PV$. Hence, for every $w\in W$, we have $\varphi^{-1}(w',w)\wedge V_p(w)\leqslant V_p ' (w')$, and consequently, we get $\varphi^{-1} \circ V_p (w')\leqslant V_p ' (w')$, which is our condition \eqref{forward simulation -3}.
Analogously, from (ib), we get $\varphi\circ V_p ' \leqslant V_p$ (the converse of \eqref{forward simulation -3}). Hence, we get the condition \eqref{forward bisimulation -3}.

Also, from (ia), we have $V(w,p)\leqslant V'(w',p) \wedge\varphi(w',w)^{-1}$, and consequently, we get $V_p\leqslant V_p ' \circ \varphi^{-1}$, which is our condition \eqref{forward simulation -1}. Analogously, from (ib), we get $V_p'\leqslant V_p \circ \varphi$ (the converse of \eqref{forward simulation -1}). Hence, we get the condition \eqref{forward bisimulation -1}. Therefore, the base condition (i) corresponds to the conditions \eqref{forward bisimulation -3} and \eqref{forward bisimulation -1}.

The condition (ii) can be rewritten as $\varphi^{-1}(w',w)\wedge R(w,v)$; then there exists $v'$ (in $\mathfrak{M}'$) such that $R'(w', v')\wedge \varphi^{-1} (v',v)$. Using a relational composition, we get $\varphi^{-1}\circ R\leqslant R' \circ \varphi^{-1}$ which is our condition \eqref{forward simulation -2}. Analogously, the condition (iii) can be written in the form $\varphi \circ R' \leqslant R\circ \varphi$. Therefore, conditions (ii) and (iii) correspond to the condition \eqref{forward bisimulation -2}.

In our work, we use duality and relax the conditions, which provides us with a plethora of different types of (pre)simulations and (pre)bisimulations. As we will see, a non-empty fuzzy relation does not have to satisfy both conditions ($\theta$-1) and ($\theta$-3), and that is why we also consider presimulations and prebisimulations. In addition, prebisimulations are important because they can give us a ``measure" of how modally equivalent some sets of formulae are.

The following example on the standard G\" odel modal logic over $[0,1]$ clarifies different types of (pre)simulations/(pre)bisimulations. From now on, for any $\theta \in \{fs, bs, fb, bb, fbb, bfb, rb\}$, by $\varphi^\theta $ we will denote the greatest simulation/bi\-sim\-u\-la\-tion of type $\theta$ between two given fuzzy Kripke models if it exists. On the other hand, by $\varphi^\theta_{*}$ we will denote the greatest fuzzy relation satisfying ($\theta $-2) and ($\theta$-3).

\begin{example}\label{Example - first example}
Let $\mathfrak{M}=(W, R, V)$ and $\mathfrak{M}'=(W', R', V')$ be two fuzzy Kripke models over the G\" odel structure, where $W=\{u,v,w\}$, $W'=\{u', v', w'\}$. Fuzzy relations $R$, $R'$ and fuzzy sets $V_p$ and $V_p '$ are represented by the following fuzzy matrices and column vectors:
\begin{align*}
&R=\begin{bmatrix}
   1 & 0.8 & 0.9 \\
   0.2 & 0.3 & 0.7 \\
   0.9 & 1 & 0.4
  \end{bmatrix}\mbox{,}\quad
  V_p=\begin{bmatrix}
    0.8 \\
    0.4\\
    0.2
  \end{bmatrix}\mbox{,}\\
&R'=\begin{bmatrix}
   0.9 & 0.8 & 1 \\
   0 & 0.3 & 0.7 \\
   1 & 0.8 & 0.4
  \end{bmatrix}\mbox{,}\quad
  V_p'=\begin{bmatrix}
    0.8 \\
    0.4\\
    0.2
  \end{bmatrix}\mbox{.}
\end{align*}

Then, using algorithms for testing the existence and computing the greatest (pre)sim\-u\-lation between fuzzy Kripke models $\mathfrak{M}$ and $\mathfrak{M}'$ from \cite{Stankovic2021}, we have:
\[
\varphi^{fs}=\begin{bmatrix}
   0.9 & 0.3 & 0.2\\
   1 & 1 & 0.2 \\
   0.9 & 0.3 & 1
  \end{bmatrix}\mbox{.}
\]
Let us verify the condition \eqref{forward bisimulation -3}: 
\[
{\varphi^{fs}}^{-1}\circ V_p=\begin{bmatrix}
   0.9 & 1 & 0.9\\
   0.3 & 1 & 0.3 \\
   0.2 & 0.2 & 1
  \end{bmatrix}\circ \begin{bmatrix}
    0.8 \\
    0.4  \\
    0.2
  \end{bmatrix} = \begin{bmatrix}
    0.8 \\
    0.4  \\
    0.2
  \end{bmatrix}=V_p '\mbox{.}
\]
Now, let us verify the condition \eqref{forward bisimulation -2}. First, we compute:
\[
{\varphi^{fs}}^{-1}\circ R = \begin{bmatrix}
   0.9 & 1 & 0.9\\
   0.3 & 1 & 0.3 \\
   0.2 & 0.2 & 1
  \end{bmatrix} \circ \begin{bmatrix}
   1 & 0.8 & 0.9 \\
   0.2 & 0.3 & 0.7 \\
   0.9 & 1 & 0.4
  \end{bmatrix} = \begin{bmatrix}
   0.9 & 0.9 & 0.9 \\
   0.3 & 0.3 & 0.7 \\
   0.9 & 1 & 0.4
  \end{bmatrix}\mbox{.}
\]
On the other hand,
\[R'\circ {\varphi^{fs}}^{-1}=
\begin{bmatrix}
   0.9 & 0.8 & 1 \\
   0 & 0.3 & 0.7 \\
   1 & 0.8 & 0.4
  \end{bmatrix}\circ \begin{bmatrix}
   0.9 & 1 & 0.9\\
   0.3 & 1 & 0.3 \\
   0.2 & 0.2 & 1
  \end{bmatrix} = 
  \begin{bmatrix}
   0.9 & 1 & 1\\
   0.3 & 0.3 & 0.7 \\
   0.9 & 1 & 0.9
  \end{bmatrix}\mbox{,}
\]
and therefore the condition \eqref{forward simulation -2} holds since ${\varphi^{fs}}^{-1}\circ R\leqslant R'\circ {\varphi^{fs}}^{-1}$. Checking condition \eqref{forward simulation -1} is formality, so we omit it.

Then, using algorithms for testing the existence and computing the greatest (pre)sim\-u\-lations and (pre)bisimulations between fuzzy Kripke models $\mathfrak{M}$ and $\mathfrak{M}'$ from \cite{Stankovic2021}, we have:
\begin{align*}
&\varphi^{bs}=\begin{bmatrix}
   0.9 & 0.4 & 0.2 \\
   1 & 0.8 & 0.2 \\
   1 & 0.8 & 1
  \end{bmatrix}\mbox{,} \quad \varphi^{fb}=\begin{bmatrix}
   0.8 & 0.3 & 0.2\\
   0.3 & 1 & 0.2 \\
   0.2 & 0.2 & 0.8
 \end{bmatrix}\mbox{,}\\
&\varphi^{bb}=\begin{bmatrix}
   0.9 & 0.4 & 0.2 \\
   0.4 & 0.8 & 0.2 \\
   0.2 & 0.2 & 0.9
  \end{bmatrix}\mbox{,} \quad \varphi^{fbb}=\begin{bmatrix}
   0.8 & 0.3 & 0.2 \\
   0.4 & 1 & 0.2 \\
   0.2 & 0.2 & 0.8
 \end{bmatrix}\mbox{,}\\
&\varphi^{bfb}=\begin{bmatrix}
   0.9 & 0.4 & 0.2 \\
   0.3 & 0.8 & 0.2 \\
   0.2 & 0.2 & 0.9
  \end{bmatrix}\mbox{,} \quad \varphi^{rb}=\begin{bmatrix}
   0.8 & 0.3 & 0.2 \\
   0.3 & 0.8 & 0.2 \\
   0.2 & 0.2 & 0.8
  \end{bmatrix}\mbox{.}
\end{align*}
In this particular example, all presimulations and prebisimulations $\varphi^\theta_{*}$ for $\theta \in \{fs, bs, fb, bb, fbb, bfb, rb\}$ satisfy the condition ($\theta$-1).
\end{example}

\subsection{Weak simulations and bisimulations}

The motivation for the introduction of weak simulations and bisimulations can be found in the theory of fuzzy automata (cf.~\cite{Jancic2014}). It has been shown that the existence of a weak simulation between two automata implies a language inclusion between them while the existence of a weak bisimulation implies language-equivalence.

Thus, we will define weak simulations and bisimulations to examine formulae inclusion and formulae-equivalence between two fuzzy Kripke models. To make the definitions of weak simulations and bisimulations as general as possible, we will define them on a set of some formulae (not necessarily on the set of all formulae). Also, the question arises as to the relationship between strong bisimulations and weak bisimulations for some fragments of logic defined in Section \ref{Section - Fuzzy Multimodal Logics}.

\begin{definition}\label{Definition - weak simulation}
Let $\mathfrak{M}=(W, \{R_{i}\}_{i\in I}, V)$ and $\mathfrak{M}'=(W', \{R_{i}'\}_{i\in I}, V')$ be two fuzzy Kripke models, let $\Psi \subseteq \PhiIH$ be a non-empty set of formulae and let $\varphi \in \mathscr{R}(W,W')$ be a non-empty fuzzy relation. We call $\varphi$ a \textit{weak forward simulation for the set $\Psi$\/} if it is a solution to the system of fuzzy relation inequalities:
\begin{align}
& V_p \leqslant V_p '\circ \varphi^{-1}\mbox{,}\qquad \mbox{for every } p\in PV\mbox{,}\label{weak simulation -1} \tag{\textit{ws-1}}\\
& \varphi^{-1}\circ V_A \leqslant V_A'\mbox{,}\qquad \mbox{for every } A\in \Psi \mbox{,} \label{weak simulation -2} \tag{\textit{ws-2}}
\end{align}
and a \textit{weak forward presimulation for the set $\Psi$} if it satisfies the condition \eqref{weak simulation -2}.
\end{definition}

We call $\varphi$ a \textit{weak backward simulation for the set $\Psi$} if it satisfies
\begin{align}
& V_p \leqslant \varphi\circ V_p ' \mbox{,}\qquad \mbox{for every } p\in PV\mbox{,}\\
&V_A\circ \varphi \leqslant V_A'\mbox{,}\qquad \mbox{for every } A\in \Psi\label{weak backward simulation-2}\mbox{,}
\end{align}
and a \textit{weak backward presimulation for the set $\Psi$} if it satisfies the condition \eqref{weak backward simulation-2}. According to Lemma \ref{Lemma - equality of compositions}, the concepts of \textit{weak forward (pre)simulation} and \textit{weak backward (pre)simulation for the set $\Psi$} mutually coincide, and we will simply call it a \textit{weak (pre)simulation}.

\begin{definition}
We call $\varphi$ a \textit{weak bisimulation for the set $\Psi$} if both $\varphi$ and $\varphi^{-1}$ are weak simulations for the set $\Psi$, i.e., if $\varphi$ satisfies
\begin{align}
& V_p \leqslant V_p '\circ \varphi^{-1}\mbox{,} \quad V_p ' \leqslant V_p \circ \varphi\mbox{,}\qquad\mbox{for every } p\in PV\mbox{,}\label{weak bisimulation-1}\tag{\textit{wb-1}}\\
& \varphi^{-1}\circ V_A \leqslant V_A'\mbox{,}\quad \varphi\circ V_A' \leqslant V_A\mbox{,}\qquad \mbox{for every } A\in \Psi\mbox{,}\label{weak bisimulation-2}\tag{\textit{wb-2}}
\end{align}
and $\varphi$ is called a \textit{weak prebisimulation for the set $\Psi$} if both $\varphi$ and $\varphi^{-1}$ are weak presimulations for the set $\Psi$, i.e., if $\varphi$ satisfies \eqref{weak bisimulation-2}.
\end{definition}

It is also possible to define four types of weak (pre)bi\-sim\-u\-lations, but they all mutually coincide.

As already mentioned, the meanings of weak simulations and bisimulations can best be explained in the case when $\mathfrak{M}$ and $\mathfrak{M}'$ are crisp Kripke models and $\varphi $ is an ordinary crisp (Boolean-valued) binary relation. The condition \eqref{weak simulation -1} is the same as \eqref{forward simulation -1}. The condition \eqref{weak simulation -2} is very similar to the condition \eqref{forward simulation -3}, but it does not refer only to the propositional variables but to all the formulae from the set $\Psi$. Hence, \eqref{weak simulation -2} means that if $w'$ simulates $w$ and the valuation $V$ assigns the value ``true'' to the formula $A\in \Psi$ in the world $w$, then the valuation $V'$ also assigns to this formula the value ``true'' in the world $w'$.

\begin{remark}\label{remark on which conditions strong and weak bisimulations coincide}
When $\Psi=PV$ then the condition \eqref{weak bisimulation-2} becomes
$$
\varphi^{-1}\circ V_p \leqslant V_p'\mbox{,} \quad \varphi\circ V_p' \leqslant V_p\mbox{,} \qquad \mbox{ for every } p\in PV\mbox{,}
$$
which is equivalent to the ($\theta$-3) condition for $\theta \in \{fb,bb,fbb,bfb,rb\}$ using \eqref{equality of composition V_p and varphi-1} and \eqref{equality of composition V_p and varphi-2}.

In this way, the condition ($\theta$-3) is packed in the condition \eqref{weak bisimulation-2} and with \eqref{weak bisimulation-1}, it can be said that the concepts of strong bisimulations and weak bisimulations coincide on conditions ($\theta$-1) and ($\theta$-3) for $\theta \in \{fb,\allowbreak bb,\allowbreak fbb,\allowbreak bfb,\allowbreak rb\}$ when $PV\subseteq \Psi$.
\end{remark}
Whether the definitions of weak (pre)simulations and (pre)bisimu\-la\-tions refer to the arbitrary set of formulae $\Psi$ or not, we usually want $\Psi$ to contain all propositional variables and we also usually take some fragments of $\PhiIH$ for the set $\Psi$.

\begin{remark}\label{Remark - equivalent form of wb-2}
Note that the condition \eqref{weak simulation -2} can be written down in an equivalent form:
\begin{equation}
\varphi(w,w')\leqslant\bigwedge_{A\in \Psi} V_A (w)\rightarrow V_A'(w')\mbox{,}\label{equivalent form of ws-2}
\end{equation}
for any $w\in W$ and $w'\in W'$. Hence, the greatest weak presimulation for the set $\Psi$ is
\begin{equation}\label{maximal weak presimulation}
\varphi_*^{ws}(w,w')=\bigwedge_{A\in \Psi} V_A (w)\rightarrow V_A'(w')\mbox{,}
\end{equation}
for any $w\in W$ and $w'\in W'$. Therefore, the greatest weak presimulation between two fuzzy Kripke models $\mathfrak{M}$ and $\mathfrak{M}'$ can be interpreted as a measure of degrees of formulae inclusion between two fuzzy Kripke models on the set $\Psi$.

In particular, if $\varphi_*^{ws}(w,w')=t$, the value $t$ can be interpreted as a measure of formulae inclusion between worlds $w$ and $w'$ on the set $\Psi$.

On the other hand, the condition \eqref{weak bisimulation-2} can be written down in an equivalent form:
\begin{equation}
\varphi(w,w')\leqslant\bigwedge_{A\in \Psi} V_A (w)\leftrightarrow V_A'(w')\mbox{,}\label{equivalent form of wb-2}
\end{equation}
for any $w\in W$ and $w'\in W'$. Hence, the greatest weak prebisimulation for the set $\Psi$ is
\begin{equation}\label{maximal weak prebisimulation}
\varphi_*^{wb}(w,w')=\bigwedge_{A\in \Psi} V_A (w)\leftrightarrow V_A'(w')\mbox{,}
\end{equation}
for any $w\in W$ and $w'\in W'$. Therefore, the greatest weak prebisimulation between two fuzzy Kripke models $\mathfrak{M}$ and $\mathfrak{M}'$ can be interpreted as a measure of degrees of formulae equality on the set $\Psi$, i.e., a measure of how much fuzzy Kripke models are $\Psi$-equivalent.

In particular, if $\varphi_*^{wb}(w,w')=t$, the value $t$ can be interpreted as a measure of formulae equality between worlds $w$ and $w'$ on the set $\Psi$.
\end{remark}

Therefore, we can conclude that a weak prebisimulation is a fuzzified version of formulae equivalence. It is generally known that a weak bisimulation on some structures is a fuzzy equivalence called \textit{weak bisimulation equivalence} and this concept is widely used in formal verification and model checking. Weak bisimulation equivalences provide better state reductions of the model than the ordinary strong bisimulations while at the same time they preserve the semantic properties of the model.

The set of weak (bi)simulations between two models is closed under an arbitrary union. Furthermore, the composition of two weak (bi)simulations is also weak (bi)simulation; a similar result applies to the weak ones, and to strong (bi)sim\-u\-la\-tions (cf.~\cite{Stankovic2021}). Therefore, we state the following lemma that can easily be proved.

\begin{lemma}
\begin{itemize}
\item[{\rm (a)}] If $\{\varphi_\alpha\}_{\alpha \in Y}$ are weak simulations/bisimulations between models $\mathfrak{M}$ and $\mathfrak{M}'$, then $\bigvee_{\alpha\in Y}\varphi_\alpha $ is also a weak simulation/bisi\-mu\-la\-tion between these models.
\item[{\rm (b)}] If $\varphi_1$ is a weak simulation/bisimulation between models $\mathfrak{M}$ and $\mathfrak{M}'$ and $\varphi_2$ is a weak simulation/bisimulation between models $\mathfrak{M}'$ and $\mathfrak{M}''$, then $\varphi_1\circ\varphi_2$ is a weak simulation/bisimulation between $\mathfrak{M}$ and $\mathfrak{M}''$.
\item[{\rm (c)}] The assertions {\rm (a)} and {\rm (b)} remain valid if the terms simulation and bisimulation are replaced by presimulation and prebisimulation, respectively.
\end{itemize}
\end{lemma}

In \cite{Stankovic2021}, the duality between forward and backward simulations, forward and backward bisimulations, and backward-forward and forward-backward bisimulations was discussed with respect to the fuzzy Kripke model and the corresponding reverse fuzzy Kripke model. A similar duality can be defined for the sets of formulae for Kripke models.

\begin{definition}
Let a mapping $\Psi \mapsto \Psi^d$ from the set
\begin{equation}\label{equation - set of sets of formulae}
\{\PhiIHPF, \PhiIHplus, \PhiIHminus, \PhiIH\}
\end{equation}
into itself be defined as follows:
\[
\left(
\begin{matrix}
\PhiIHPF & \PhiIHplus & \PhiIHminus & \PhiIH\\
\PhiIHPF & \PhiIHminus & \PhiIHplus & \PhiIH
\end{matrix}
\right)\mbox{,}
\]
where $\PhiIHPF$ denotes the set of propositional formulae, i.e., formulae without any modal operators.
\end{definition}

\begin{example}
Let us note that the set \eqref{equation - set of sets of formulae} can contain arbitrary dual sets of formulae. That fact follows from the reversing duality of modal operators for every $i\in I$:
\[
\left(
\begin{matrix}
\square_i & \square_i^{-} & \lozenge_i & \lozenge_i^{-}\\
\square_i^{-} & \square_i & \lozenge_i^{-} & \lozenge_i
\end{matrix}
\right)\mbox{.}
\]
For example, let $\Psi$ contain set $\Phi=\{ A\wedge B, \lozenge_i A, \square_j ^{-}A \rightarrow \lozenge_k B\}$. Then, $\Psi$ also contains $\Phi^{d}=\{A\wedge B,\lozenge_i ^{-} A, \square_j A \rightarrow \lozenge_k ^{-} B\}$, for some $i,j,k\in I$, such that $\Psi(\Phi)=\Phi^{d}$.
\end{example}

Now we can state the following proposition:

\begin{proposition}\label{Proposition - Duality for sets of formulae}
Let $\mathfrak{M}=(W, \{R_{i}\}_{i\in I}, V)$ and $\mathfrak{M}'=(W', \{R_{i}'\}_{i\in I}, V')$ be two fuzzy Kripke models, let $\mathfrak{M}^{-1}$ and $\mathfrak{M'}^{-1}$ be the reverse fuzzy Kripke models for $\mathfrak{M}$ and $\mathfrak{M}'$, respectively, and let $\Psi \in \{\PhiIHPF, \PhiIHplus, \PhiIHminus, \PhiIH\}$.

Then the following is true:
\begin{itemize}
\item[{\rm (a)}] $\varphi$ is a weak simulation/bisimulation for set $\Psi$ between $\mathfrak{M}$ and $\mathfrak{M}'$ if and only if $\varphi$ is a weak simulation/bisimulation for the set $\Psi^d$ between the reverse fuzzy Kripke models $\mathfrak{M}^{-1}$ and $\mathfrak{M'}^{-1}$.
\item[{\rm (b)}] The assertion {\rm (a)} remains valid if the terms simulation and bisimulation are replaced by a presimulation and prebisimulation, respectively.
\end{itemize}
\end{proposition}
\begin{proof}
The proof is directly follows from the definition of formulae, the definitions of the sets of formulae and the reverse model.
\end{proof}

\section{Hennessy-Milner type theorems for fuzzy multimodal logics}\label{Section - Hennessy-Milner type theorems for fuzzy multimodal logics}

Let us briefly recall the essence of the original Hennessy-Milner theorem. The condition {\rm (i)} from Definition \ref{Definition - bisimulation from Blackburn} means that bisimilar propositional variables have the same properties (values). The conditions {\rm (ii)} and {\rm (iii)} ensure that the relations of the models are sufficiently similar to ensure the preservation of truth of the formulae.

Therefore, a bisimulation preserves the truth values of the formulae. Hence, for a basic modal language PML$^{+}$, bisimilar worlds are formulae equivalent with respect to the set of all formulae.

The converse of this assertion, meaning that if worlds are formulae equivalent, they must be bisimilar, generally does not hold, but it is valid for some classes of Kripke models. This is exactly what the Hennessy-Milner theorem specifies.

\begin{theorem}[Hennessy-Milner Theorem]
Let $\mathfrak{M}=(W, R, V)$ and $\mathfrak{M}'=(W', R', V')$ be two image-finite Kripke models over the basic modal language {\rm PML}$^{+}$. Then, for any $w\in W$ and $w'\in W'$, $w$ and $w'$ are bisimilar with respect to {\rm PML}$^{+}$ if and only if $w$ and $w'$ are {\rm PML}$^{+}$-equivalent.
\end{theorem}

In other words, the Hennessy-Milner Theorem implies that the two worlds $w$ and $w'$ are bisimilar with respect to PML$^{+}$ if and only if the sets of PML$^{+}$-formulae valid in $w$ and $w'$ coincide. In the context of fuzzy multimodal logics, we can make the following generalization:

Let $\mathfrak{M}=(W, \{R_i\}_{i\in I}, V)$ be a fuzzy Kripke model and let $\Psi \subseteq \PhiIH$ be some set of formulae. For each $w\in W$ we define a fuzzy subset $V_w$ of $\Psi $ by $V_w(A)=V(w,A)$, for every $A\in \Psi $. This means that the degree to which a formula $A$ belongs to the fuzzy set $V_w$ is equal to the truth degree of $A$ in the world $w$. In classical modal logic, $V_w$ is simply the set of all formulae valid in the world $w$, so in the context of fuzzy modal logic we will say that $V_w$ is the fuzzy set of formulae that are true (with a certain degree of truth) in $w$.

Now, let $\mathfrak{M}=(W, \{R_{i}\}_{i\in I}, V)$ and $\mathfrak{M}'=(W', \{R_{i}'\}_{i\in I}, V')$ be two fuzzy Kripke models and let $\Psi \subseteq \PhiIH$ be some set of formulae. As we noted in the previous section, the greatest weak (pre)bisimulation for the set $\Psi$ (when it exists) is given by
\[
\varphi(w,w')=\bigwedge_{A\in\Psi} V_A(w)\leftrightarrow V'_A(w')=
\bigwedge_{A\in\Psi} V_w(A)\leftrightarrow V'_{w'}(A).
\]
In the fuzzy set theory, the expression far right in this equation is known as the degree of equality of fuzzy sets $V_w$ and $V'_{w'}$, and therefore, the greatest weak prebisimulation is the measure of the degree of equality of fuzzy sets of formulae from $\Psi $ valid in two worlds $w$ and $w'$, that is, the measure of the degree of modal equivalence between worlds $w$ and $w'$ with respect to formulae from $\Psi $.

Note that the Hennessy-Milner theorem replaces weak bisimulations by bisimulations, which is important because the greatest bisimulations between finite models can be computed by algorithms of polynomial complexity; in contrast to the greatest weak bisimulations, which are generally computed by algorithms of exponential complexity. An even bigger problem arises when computing the greatest fuzzy weak bisimulations.

Our aim is to prove three Hennessy-Milner type theorems for fuzzy multimodal logics over linearly ordered Heyting algebras. We will show that the degree of modal equivalence with respect to plus-formulae, between two worlds in image-finite Kripke models, can be expressed by the greatest forward (pre)bi\-sim\-u\-la\-tion, the degree of modal equivalence with respect to mi\-nus-for\-mulae, between two worlds in domain-finite Kripke models, can be expres\-sed by the greatest backward (pre)bi\-sim\-u\-la\-tion, and the degree of modal equivalence with respect to all formulae, between two worlds in degree-finite Kripke models, can be expressed by the greatest regular (pre)bi\-sim\-u\-la\-tion.

First we prove the following theorem.

\begin{theorem}[The Hennessy-Milner type theorem for plus-for\-mulae]\label{Hennessy-Milner Theorem for plus-formulae}
Let $\mathfrak{M}=(W, \{R_{i}\}_{i\in I}, V)$ and $\mathfrak{M}'=(W', \{R_{i}'\}_{i\in I}, V')$ be two image-finite fuzzy Kripke models over a linearly ordered Heyting algebra $\mathscr H$. The greatest weak $\PhiIHplus$-pre\-bisimulation (resp.~the greatest $\PhiIHplus$-bisimulation) between $\mathfrak{M}$ and $\mathfrak{M}'$, if it exists, is the greatest forward prebisim\-ulation (resp.~the greatest forward bisimulation) between $\mathfrak{M}$ and $\mathfrak{M}'$.
\end{theorem}

The proof is based on the next two lemmas.

\begin{lemma}\label{lemma - first lemma}
Under the assumptions of Theorem \ref{Hennessy-Milner Theorem for plus-formulae}, any forward prebisimulation (resp.~forward bisimulation) between $\mathfrak{M}$ and $\mathfrak{M}'$ is a weak $\PhiIHplus$-prebisimulation (resp.~$\PhiIHplus$-bisimulation) between $\mathfrak{M}$ and $\mathfrak{M}'$.
\end{lemma}

\begin{proof}
Let $\varphi $ be a forward prebisimulation between $\mathfrak{M}$ and $\mathfrak{M}'$. To prove that $\varphi $ is a weak $\PhiIHplus$-prebisimulation we will prove that
\begin{equation}\label{eq:to.be.proved}
\varphi(u,u')\leqslant V_{A}(u) \leftrightarrow V_{A}'(u'),\ \
\end{equation}
for all $u\in W$, $u'\in W'$ and every $A\in \PhiIHplus$. This will be proved by induction on the complexity of a formula $A$.

\textit{Induction basis:} If $A=p\in PV$, then from the fact that $\varphi$ is forward bisimulation we have that $\varphi^{-1}\circ V_p \leqslant V_{p}'$ and $\varphi\circ V_{p}' \leqslant V_p$, which means that
\[
\varphi^{-1}(u',u)\wedge V_p (u)\leqslant V_{p}'(u'), \quad \varphi(u,u')\wedge V_p' (u')\leqslant V_{p}(u),
\]
for all $u\in W$, $u'\in W'$ and $p\in PV$. Using the adjunction property (\ref{adjunction property in Heyting algebra}) we get
\[
\varphi(u,u') \leqslant V_p (u) \rightarrow V_{p}'(u'), \quad \varphi(u,u') \leqslant V_{p}'(u') \rightarrow V_p (u),
\]
and therefore,
\[
\varphi(u,u')\leqslant V_p (u) \leftrightarrow V_{p}'(u'),
\]
for all $u\in W$, $u'\in W'$ and $p\in PV$. Consequently, (\ref{eq:to.be.proved}) holds for any propositional variable $p$. It trivially holds for any truth constant $\overline{t}$.

\textit{Induction step:} a) Let $A=B\wedge C$ and let (\ref{eq:to.be.proved}) hold for $B$ and $C$, i.e.,
\[
\varphi(u,u')\leqslant V_B(u) \leftrightarrow V_B'(u'), \quad \varphi(u,u')\leqslant V_C(u) \leftrightarrow V_C'(u'),
\]
for all $u\in W$, $u'\in W'$. This yields
\[
\varphi(u,u')\leqslant (V_B(u) \leftrightarrow V_B'(u'))\wedge (V_C(u) \leftrightarrow V_C'(u'))\mbox{.}
\]
Using the property of Heyting algebras $(x_1 \leftrightarrow y_1) \wedge (x_2 \leftrightarrow y_2) \leqslant (x_1 \wedge x_2) \leftrightarrow (y_1 \wedge y_2)$, we get
\begin{align*}
\varphi(u,u')&\leqslant (V_B(u) \leftrightarrow V_B'(u'))\wedge (V_C(u) \leftrightarrow V_C'(u'))\\
 &\leqslant (V_B(u) \wedge V_C(u))\leftrightarrow (V_B'(u') \wedge V_C'(u'))\\
 &= V_{B\wedge C}(u) \leftrightarrow V_{B\wedge C}'(u') \\
 &= V_{A}(u)\leftrightarrow V_{A}'(u'),
\end{align*}
for all $u\in W$ and $u'\in W'$, so we conclude that (\ref{eq:to.be.proved}) holds for $A=B\land C$.

b) Let $A$ be of the form $B \rightarrow C$ and let (\ref{eq:to.be.proved}) hold for $B$ and $C$. In a similar way as in a), using the property of Heyting algebras
\[
(x_1 \leftrightarrow y_1) \wedge (x_2 \leftrightarrow y_2) \leqslant (x_1 \rightarrow x_2) \leftrightarrow (y_1 \rightarrow y_2),
\]
we prove that (\ref{eq:to.be.proved}) also holds for $A$.

c) Let $A=\lozenge_i B$ and let (\ref{eq:to.be.proved}) hold for $B$, i.e.,
\begin{align*}
\varphi(u,u')&\leqslant V_B (u) \leftrightarrow V_B'(u')\\
 &=(V_B (u) \rightarrow V_B'(u'))\land ( V_B'(u')\rightarrow V_B (u))\mbox{,}
\end{align*}
for all $u\in W$ and $u'\in W'$. Then it follows that
\[
\varphi(u,u')\leqslant (V_B (u) \rightarrow V_B'(u')), \quad
\varphi(u,u')\leqslant (V_B'(u')\rightarrow V_B (u)),
\]
and using the adjunction property (\ref{adjunction property in Heyting algebra}) we conclude
\[
\varphi^{-1}(u',u)\wedge V_B (u)\leqslant V_{B}'(u'), \quad \varphi(u,u')\wedge V_B' (u')\leqslant V_{B}(u)
\]
for all $u\in W$ and $u'\in W'$. Hence,
\[
\varphi^{-1}\circ V_B \leqslant V_{B}', \quad \varphi\circ V_{B}' \leqslant V_B,
\]
and we have
\begin{align*}
\varphi^{-1} \circ V_{A} &= \varphi^{-1} \circ R_i \circ V_{B} \leqslant R_i ' \circ \varphi^{-1} \circ V_{B}\qquad \mbox{(by (\ref{forward bisimulation -2}))}\\
&\leqslant R_i'\circ V_{B}'=V_{A}' \mbox{,}
\end{align*}
for every $i\in I$. From $\varphi^{-1} \circ V_{A} \leqslant V_{A} '$ we conclude that $\varphi(u,u')\leqslant V_{A}(u) \rightarrow V_{A}'(u')$. Thus, we conclude that $\varphi (u,u') \leqslant V_{A}'(u') \rightarrow V_{A}(u)$, for all $u\in W$ and $u'\in W'$, which means that
\[
\varphi(u,u')\leqslant V_{A}(u)\leftrightarrow V_{A}'(u'),
\]
for all $u\in W$ and $u'\in W'$. Therefore, we have proved that (\ref{eq:to.be.proved}) is also true for $A=\lozenge_i B$.

d) Suppose that $A=\square_i B$ and (\ref{eq:to.be.proved}) holds for $B$. In a similar way as in c), from $\varphi(u,u')\leqslant V_B (u) \leftrightarrow V_B'(u')$, for all $u\in W$ and $u'\in W'$, we obtain
\[
\varphi^{-1} \circ V_{B} \leqslant V_{B}', \quad \varphi \circ V_{B}' \leqslant V_{B} .
\]
Since the underlying Heyting algebra is linearly ordered, values $\varphi(u,u')=\varphi^{-1}(u',u)$, $V_A(u)$ and $V_{A}'(u')$ can be com\-pared with each other, for all $u\in W$, $u'\in W'$, there\-fore, a case analysis can be used.

If $\varphi^{-1}(u',u)\leqslant V_A(u)\wedge V_{A}(u')$ and $V_A(u)\neq V_{A}'(u')$, then
\[
\varphi(u,u')=\varphi^{-1}(u',u)\leqslant V_A(u)\wedge V_{A}' (u')=V_A(u)\leftrightarrow V_{A}' (u').
\]
In case $V_A(u)= V_{A}' (u')$ we have that $V_A(u)\leftrightarrow V_{A}' (u')=1$, which again gives $\varphi (u,u')\leqslant V_A(u)\leftrightarrow V_{A}' (u')$.

Hence, we need to consider only the case when
\[
\varphi^{-1}(u',u)>V_A(u)\wedge V_{A}' (u').
\]
Without loss of generality, we can assume that $\varphi^{-1}(u',u)>V_A(u)$, and then we have:
\begin{align}
V_A(u)&=\varphi^{-1}(u',u) \wedge V_A(u)\nonumber\\
&=\varphi^{-1}(u',u) \wedge \bigwedge_{v\in W}(R_i (u,v)\rightarrow V_B(v))\nonumber\\
&=\bigwedge_{v\in W}\left[\varphi^{-1}(u',u) \wedge\bigl(R_i (u,v)\rightarrow V_B(v)\bigr)\right]\nonumber\\
&=\bigwedge_{v\in W}\left[\varphi^{-1}(u',u) \wedge\bigl(\varphi^{-1}(u',u) \wedge R_i (u,v)\rightarrow V_B(v)\bigr)\right]\nonumber\\
&=\varphi^{-1}(u',u) \wedge\bigwedge_{v\in W}\left[\varphi^{-1}(u',u) \wedge R_i (u,v)\rightarrow V_B(v)\right]\label{last equation in align}
\end{align}
In the third and fifth lines we have used the property
\[
x\wedge \left( \bigwedge_{i\in I} y_i \right) = \bigwedge_{i\in I}(x\wedge y_i),
\]
which holds for every index set $I$ in a complete Heyting algebra. In the fourth line, we have used the well-known equation that holds in Heyting algebras
\[
x\wedge(y\rightarrow z)=x\wedge(x\wedge y\rightarrow z).
\]

According to the starting assumption, $\varphi$ is a forward prebisimulation, so it satisfies \eqref{forward bisimulation -2}, i.e.
\[
\varphi^{-1} \circ R_{i} \leqslant R_{i}' \circ \varphi^{-1}\mbox{,} \qquad \mbox{for every }i\in I\mbox{.}
\]
Next, since $R_i'$ is image-finite, for each $v\in W$ we can find $v'\in W'$ such that
\[
\varphi^{-1}(u',u) \wedge R_i (u,v) \leqslant R_i' (u',v')\wedge \varphi^{-1}(v',v),
\]
and it follows
\[
\bigl(\varphi^{-1}(u',u) \wedge R_i (u,v)\bigr)\rightarrow V_B(v) \geqslant \bigl(\varphi^{-1}(v',v) \wedge R_i' (u',v')\bigr)\rightarrow V_B(v)\mbox{.}
\]
Now, the two cases need to be analyzed. If $V_B(v)=V_B'(v')$, then
\[
R_i'(u',v') \rightarrow V_B(v)=R_i'(u',v') \rightarrow V_B'(v')\mbox{.}
\]
Since
\[
\bigl(\varphi^{-1}(v',v)\wedge R_i'(u',v')\bigr) \rightarrow V_B(v) \geqslant
 R_i'(u',v') \rightarrow V_B(v)\mbox{,}
\]
it follows
\[
\bigl(\varphi^{-1}(v',v)\wedge R_i'(u',v')\bigr) \rightarrow V_B(v) \geqslant
 R_i'(u',v') \rightarrow V_B'(v')\mbox{.}
\]
On the other hand, if $V_B(v) \neq V_B'(v')$, then by the induc\-tion hypothesis we have that
\[
\varphi^{-1}(v',v)\leqslant(V_B(v)\leftrightarrow V_B'(v'))\leqslant V_B(v).
\]
Thus,
\[
\bigl(\varphi^{-1}(v',v)\wedge R_i'(u',v')\bigr)\rightarrow V_B(v)=1 \geqslant R_i '(u',v')\rightarrow V_B'(v')\mbox{.}
\]
In both cases, we have shown that for any $v\in W$, we can find $v'$ so that
\[
\bigl(\varphi^{-1}(u',u)\wedge R_i (u,v)\bigr)\rightarrow V_B(v)\geqslant R_i'(u', v')\rightarrow V_B'(v').
\]
Therefore,
\[
\bigwedge_{v\in W} \bigl(\varphi^{-1}(u',u)\wedge R_i(u,v)\bigr)\rightarrow V_B(v)
\geqslant\bigwedge_{v'\in W'} R_i'(u',v')\rightarrow V_B'(v')
=V_{A}'(u')
\]
and using \eqref{last equation in align} we conclude:
$
V_A(u)\geqslant \varphi^{-1}(u',u)\wedge V_{A}'(u')\mbox{.}
$
Because of the assumption that $\varphi^{-1}(u',u)>V_A(u)$, we have
\[
V_A(u)\geqslant V_{A}'(u')\ \ \text{and}\ \ \varphi^{-1}(u',u)>V_{A}'(u').
\]

Analogously, by the same reasoning we can prove that $V_{A}'(u')\geqslant V_A(u)$, since $\varphi^{-1}(u',u)>V_{A}'(u')$. Hence, we have $V_A(u)=V_{A}'(v')$, and since $\varphi(u,u')=\varphi^{-1}(u',u)$ it follows
\[
\varphi(u,u')\leqslant V_A(u)\leftrightarrow V_{A}'(u')=1
\]
when $\varphi^{-1}(u',u)>V_A(u)\wedge V_{A}'(u')$.

This completes the proof of the statement that every forward prebisimulation is a weak $\PhiIHplus$-prebi\-simulation. This also means that every forward bisimulation is a weak $\PhiIHplus$-bisimulation, since the additional condi\-tions (\textit{fb-1}) and (\textit{wb-1}) that distinguish between prebisimulations and bisimulations are the same in both cases.
\end{proof}

\begin{lemma}\label{lemma - second lemma}
Under the assumptions of Theorem \ref{Hennessy-Milner Theorem for plus-formulae}, the greatest weak $\PhiIHplus$-prebisimulation (resp.~the greatest $\PhiIHplus$-bisimulation) between $\mathfrak{M}$ and $\mathfrak{M}'$, if it exists, is a forward prebisimulation (resp.~a forward bisimulation) between $\mathfrak{M}$ and $\mathfrak{M}'$.
\end{lemma}

\begin{proof}
Let $\varphi $ be a weak $\PhiIHplus$-prebisimulation. According to Remark \ref{remark on which conditions strong and weak bisimulations coincide}, $\varphi$ satisfies the condition \eqref{forward bisimulation -3}. Hence, it remains to prove that \eqref{forward bisimulation -2} is true.

To prove that, we will use the proof by a contradiction and the same method used in Lemma 2 from \cite{Fan2015}. Namely, we will prove the assumption that (\textit{wb-2}) is true while (\textit{fb-2}) is not true, which leads to a contradiction. Therefore, let us assume that (\textit{fb-2}) does not hold. This means that there exists $i\in I$ so that
\begin{equation}\label{eq:2cases}
\varphi^{-1}\circ R_i\not\leqslant R_i'\circ\varphi^{-1}\qquad \text{or}\qquad \varphi\circ R_i'\not\leqslant R_i\circ\varphi\mbox{,}
\end{equation}
for some $i\in I$. We will consider only the case
\begin{equation}\label{eq:1.case}
\varphi^{-1}\circ R_i\not\leqslant R_i'\circ\varphi^{-1} ,
\end{equation}
for some $i\in I$, because the second case in (\ref{eq:2cases}) can be treated similarly. By the hypothesis, the under\-lying Heyting alge\-bra $\mathscr H$ is linearly ordered, so the formula (\ref{eq:1.case}) means that there are $u,v\in W$ and $u'\in W'$ such that
\begin{equation}\label{eq:greater}
\varphi^{-1}(u',u)\land R_i(u,v)>\bigvee_{v'\in W'}R_i'(u',v')\land\varphi^{-1}(v',v).
\end{equation}
Let $W_{u'}'=\{v'\in W'\mid R_i'(u',v')>0\}$. By the assumption of the theorem, $R_i'$ is image-finite, which means that $W_{u'}'$ is finite.

To simplify, let us set
\begin{align*}
    &x=\varphi^{-1}(u',u), \qquad  y=R_i(u,v),\\
    &x_{v'}=R_i'(u',v'),\ \qquad y_{v'}=\varphi^{-1}(v',v),
\end{align*}
for each $v'\in W_{u'}'$. Then, the formula (\ref{eq:greater}) becomes
\begin{equation}\label{eq:greater2}
x\land y>\bigvee_{v'\in W_{u'}'}x_{v'}\land y_{v'}.
\end{equation}
Due to (\ref{eq:greater2}), for each $v'\in W_{u'}'$ we have that $x_{v'}\land y_{v'}<x\land y$, and because of the linearity of the ordering in $\mathscr H$, we get that either $x_{v'}<x\land y$ or $y_{v'}<x\land y$.

\underline{\textit{Case $y_{v'}<x\land y$:}} If $y_{v'}<x\land y$, i.e.,
\[
\varphi^{-1}(v',v)=\varphi(v,v')<x\land y,
\]
then by the definition of $\varphi=\varphi_{*}^{wb}$, for each $v'\in W_{u'}'$ there exists $A_{v'} \in \PhiIHplus$ such that
\[
(V(v, A_{v'})\leftrightarrow V' (v', A_{v'}))<x\wedge y.
\]
In fact, since the underlying algebra $\mathscr{H}$ is linearly ordered, then $(V(v, A_{v'})\leftrightarrow V' (v', A_{v'}))=V(v, A_{v'})\wedge V' (v', A_{v'})$, for $V(v, A_{v'})\neq V' (v', A_{v'})$ and then $A_{v'}$ can be any formula such that $V(v, A_{v'})< x\wedge y$ or $V' (v', A_{v'})< x\wedge y$.

Set $z_{v'}=V(v, A_{v'})$. Now we define $B_{v'}$, for each $v'\in W_{u'}'$, as follows:
\begin{equation}\label{definition of formula -B_u'}
 B_{v'} =
  \begin{cases}
    \ \overline{1}, & \text{if } x_{v'} <x\wedge y \\
    \ A_{v'} \leftrightarrow \overline{z_{v'}}, & \text{otherwise }
  \end{cases}
\end{equation}
Note that if $x_{v'}\geqslant x\wedge y$, then we have that
\begin{align*}
V'(v', B_{v'})&= V'(v', A_{v'} \leftrightarrow \overline{z_{v'}})\\
&=V'(v', A_{v'})\leftrightarrow V(v, A_{v'})<x\wedge y
\end{align*}
and
\begin{align*}
V(v, B_{v'})&=V(v, A_{v'}\leftrightarrow \overline{z_{v'}})\\
&=V(v, A_{v'})\leftrightarrow V(v, A_{v'})=1\mbox{.}
\end{align*}
Further, set $B=\bigwedge_{v'\in W_{u'}'} B_{v'}$. Then,
\begin{align*}
V'(u', \lozenge_i B)&=\bigvee_{v'\in W_{u'}'} R_i'(u',v')\wedge V'(v', B)\\
&=\bigvee_{v'\in W_{u'}'} x_{v'} \wedge V'(v', B)\mbox{.}
\end{align*}
Thus,
\begin{align*}
V'(u', \lozenge_i B)&\leqslant \left( \bigvee_{\substack{v'\in W_{u'}'\\x_{v'}< x\wedge y}} x_{v'}\right) \vee\left( \bigvee_{\substack{v'\in W_{u'}'\\x_{v'}\geqslant x\wedge y}} V'(v', B_{v'})\right)<x\wedge y\mbox{.}
\end{align*}
On the other hand,
\begin{align*}
V (u, \lozenge_i B)&=\bigvee_{v\in W} R_i(u,v)\wedge V(v,B)\\
&\geqslant R_i(u,v)\wedge V(v,B)=y\geqslant x\wedge y\mbox{.}
\end{align*}
Now, according to \eqref{equivalent form of wb-2}, we have
\begin{align*}
x=\varphi^{-1} (u',u)&\leqslant (V'(u', \lozenge_i B)\leftrightarrow V (u, \lozenge_i B))\\
&=V'(u', \lozenge_i B)\wedge V (u, \lozenge_i B)\\
&=V'(u', \lozenge_i B)<x\wedge y
\end{align*}
which represents a contradiction.

\underline{\textit{Case $x_{v'}<x\land y$:}} Set $B=\overline{1}$ ($B$ can also be any proposi\-tional formula that is a tautology, for example, $p\leftrightarrow p$).

In the same way as in the previous case, we conclude that
\[
V'(u', \lozenge_i B)<x\land y,\qquad V (u, \lozenge_i B)\geqslant y\geqslant x\land y,
\]
whence
\begin{align*}
x&=\varphi^{-1}(u',u)=V'(u', \lozenge_i B)\wedge V (u, \lozenge_i B)\\
&=V'(u', \lozenge_i B)<x\wedge y,
\end{align*}
and again we get a contradiction.

Therefore, in all cases, the assumption that \eqref{weak bisimulation-2} is true while \eqref{forward bisimulation -2} is not true leads to a contradiction, whence we finally conclude that (\textit{wb-2}) implies (\textit{fb-2}), i.e., that every weak $\PhiIHplus$-prebisimulation is a forward prebisimulation. Since the conditions (\textit{fb-1}) and (\textit{wb-1}) are the same, we also conclude that every weak $\PhiIHplus$-bisimulation is a forward bisimulation.

This completes the proof of the lemma, as well as the proof of the Theorem \ref{Hennessy-Milner Theorem for plus-formulae}.
\end{proof}

\begin{remark} Note that the proof of the Lemma \ref{lemma - second lemma} can be carried out by constructing the formula $\square_i B$ instead of $\lozenge_i B$. Here we give only the part of the proof that needs to be modified.
\end{remark}

\begin{proof}
Let $B=\bigwedge_{v'\in W_{u'}'} B_{v'}$. Then
\begin{align*}
V'(u', \square_i B)&=\bigwedge_{v'\in W_{u'}'} R_i'(u',v')\rightarrow V'(v', B)=\bigwedge_{v'\in W_{u'}'} x_{v'} \rightarrow V'(v', B)\\
&=\bigwedge_{v'\in W_{u'}'} x_{v'} \rightarrow V'\left(v', \bigwedge_{v'\in W_{u'}'} B_{v'}\right)\\
& \leqslant \bigwedge_{v'\in W_{u'}'} x_{v'} \rightarrow V'(v', B_{v'})\mbox{.}
\end{align*}
Thus,
\begin{align*}
V'(u', \square_i B)&= \left( \bigwedge_{\substack{v'\in W_{u'}'\\x_{v'}< x\wedge y}} x_{v'}\rightarrow V'(v', B_{v'}) \right) \wedge\left( \bigwedge_{\substack{v'\in W_{u'}'\\x_{v'}\geqslant x\wedge y}} x_{v'}\rightarrow V'(v', B_{v'})\right)\\
&\leqslant \left( \bigwedge_{\substack{v'\in W_{u'}'\\x_{v'}\geqslant x\wedge y}} x_{v'}\rightarrow V'(v', B_{v'})\right)=\left( \bigwedge_{\substack{v'\in W_{u'}'\\x_{v'}\geqslant x\wedge y}} V'(v', B_{v'})\right)<x\wedge y\mbox{.}
\end{align*}
On the other hand,
\begin{align*}
V (u, \square_i B)&=\bigwedge_{w\in W} R_i(u, w)\rightarrow V(w, B)\\
&\geqslant R_i(u, v)\rightarrow V(v, B)=y\rightarrow V(v, B)=1\mbox{.}
\end{align*}
Hence, by the definition of $\varphi=\varphi_{*} ^{wb}$ and \eqref{equivalent form of wb-2} we have
\begin{align*}
x=\varphi^{-1} (u',u)&\leqslant (V'(u', \square_i B)\leftrightarrow V (u, \square_i B))\\
&=(V'(u', \square_i B)\leftrightarrow 1=V'(u', \square_i B)<x\wedge y,
\end{align*}
which again represents a contradiction.
\end{proof}

In a similar way we prove the following two theorems:

\begin{theorem}[{The Hennessy-Milner type theorem for minus-for\-mulae}]\label{Hennessy-Milner Theorem for minus-formulae}
Let $\mathfrak{M}=(W, \{R_{i}\}_{i\in I}, V)$ and $\mathfrak{M}'=(W', \{R_{i}'\}_{i\in I}, V')$ be two domain-finite fuzzy Kripke models over a linearly or\-dered Heyting algebra $\mathscr H$. The greatest weak $\PhiIHminus$-prebisimulation (resp.~the greatest $\PhiIHminus$-bisimulation) between $\mathfrak{M}$ and $\mathfrak{M}'$, if it exists, is the greatest backward prebisimulation (resp.~the greatest backward bisimula\-tion) between $\mathfrak{M}$ and $\mathfrak{M}'$.
\end{theorem}

\begin{theorem}[The Hennessy-Milner type theorem for the set of all modal formulae]\label{Hennessy-Milner Theorem for all formulae}
Let $\mathfrak{M}=(W, \{R_{i}\}_{i\in I}, V)$ and $\mathfrak{M}'=(W', \{R_{i}'\}_{i\in I}, V')$ be two degree-finite fuzzy Kripke models over a linearly ordered Heyting algebra $\mathscr H$. The greatest weak $\PhiIH$-prebisimula\-tion (resp.~the greatest $\PhiIH$-bisi\-mu\-la\-tion) between $\mathfrak{M}$ and $\mathfrak{M}'$, if it exists, is the greatest regular prebisimulation (resp.~the greatest regular bisimula\-tion) between $\mathfrak{M}$ and $\mathfrak{M}'$.
\end{theorem}
Also note that the Theorem \ref{Hennessy-Milner Theorem for minus-formulae} follows from the Theorem \ref{Hennessy-Milner Theorem for plus-formulae} based on the duality between the Kripke models and their reverse models. The Theorem \ref{Hennessy-Milner Theorem for all formulae} is a direct consequence of the Theorems \ref{Hennessy-Milner Theorem for plus-formulae} and \ref{Hennessy-Milner Theorem for minus-formulae}.

\begin{remark}\label{remark - first Lemma does not hold for presimulations}
Lemma \ref{lemma - first lemma} generally does not hold in the case of pre\-simulations for some set $\Psi$, i.e., inequality $\varphi_{*}^{fs}(w,w')\leqslant \varphi_{*}^{ws} (w,w')$, does not hold.
\end{remark}

For example, if a formula $\alpha$ is of the form $A\rightarrow B$ and the result holds for $A$ and $B$, using the adjunction property \eqref{adjunction property in Heyting algebra} we have
\begin{align*}
\varphi_{*}^{fs}(w,w')&\leqslant V_A(w) \rightarrow V_A'(w')\mbox{,}\\
\varphi_{*}^{fs}(w,w')&\leqslant V_B(w) \rightarrow V_B'(w')\mbox{,}
\end{align*}
for every $w\in W$ and $w'\in W'$. Hence, we have
$$
\varphi_{*}^{fs}(w, w')\leqslant (V_A (w)\rightarrow V_A'(w'))\wedge (V_B(w) \rightarrow V_B'(w'))\mbox{.}
$$

But, we want to prove $\varphi(w,w')\leqslant (V_A (w)\rightarrow V_B(w))\wedge (V_A' (w')\rightarrow V_B'(w'))$ and for that, we need the property
$(x_1\rightarrow y_1)\wedge (x_2\rightarrow y_2)\leqslant (x_1\rightarrow x_2)\wedge (y_1\rightarrow y_2)\mbox{,}$
which simply does not hold in the linearly ordered Heyting algebra. To make sure, we can take a G\"{o}del $[0,1]$ structure and the following values, $x_1=0.7$, $y_1=0.8$, $x_2=0.6$ and $y_2=0.7$.

However, this does not mean that the Hennessy-Milner property is not valid for fuzzy simulations in another logic. For example, in \cite{Nguyen2021} the Hen\-nessy-Milner property for fuzzy simulations was given for Fuzzy Labelled Transition Systems in Fuzzy Propositional Dynamic Logic.

\section{Hennessy-Milner type theorems for Propositional Modal Logics}\label{Section - Hennessy-Milner type theorems for Propositional Modal Logics}

The given definitions of simulations and bisimulations as well as the Hen\-nessy-Milner type theorems also apply in special cases, such as Propositional Modal Logic. We expand basic modal language (see, for example, \cite{Blackburn2001}) with inverse modal operators.

\begin{definition}
Let $\mathscr{B}=(B, \wedge, \vee, \rightarrow, 0, 1)$ be a two-element Boolean algebra and write $\overline{B}=\{\overline{t}\mid t\in B\}$ for the elements of $\mathscr{B}$ viewed as constants. Define the language $\Phi_{\mathscr{B}}$ via the grammar
\[
A::=\overline{t}\mid p\mid A\wedge A\mid A\rightarrow A\mid \lozenge A\mid \lozenge^{-} A
\]
where $\overline{t}\in \overline{B}$ and $p$ ranges over some set $PV$ of proposition letters.
\end{definition}
We also use standard abbreviations:
\begin{flalign*}
&\neg A \equiv A\rightarrow \overline{0}\mbox{ (\textit{negation})},&&\\
&A \leftrightarrow B \equiv (A\rightarrow B) \wedge (B\rightarrow A)\mbox{ (\textit{equivalence}),}&&\\
&A\vee B \equiv \neg (\neg A \wedge \neg B)\mbox{ (\textit{disjunction}),}&&\\
&\square A \equiv \neg \lozenge \neg A\mbox{ (\textit{necessity operator}),}&&\\
&\square^{-} A \equiv \neg \lozenge^{-} \neg A\mbox{ (\textit{inverse necessity operator}).}&&
\end{flalign*}

Let PML$^+$ be the set of all formulae with modality $\lozenge$ and its dual operator $\square$, PML$^-$ be the set of all formulae for propositional modal logics with converse modality $\lozenge^-$ and its dual operator $\square^-$. Finally, let PML denotes the set of all formulae for propositional modal logic with modalities $\lozenge$ and $\lozenge^-$ and with their dual operators $\square$ and $\square^-$, respectively. 

Now, using the fact that a weak (pre)bisimulation is logical equivalence on a set of formulae, then the Hennessy-Milner theorems can be reformulated as follows:

\begin{theorem}[The Hennessy-Milner theorem for PML$^+$]\label{Hennessy-Milner Theorem for PML+}
Let $\mathfrak{M}=(W, R, V)$ and $\mathfrak{M}'=(W', R', V')$ be two image-finite {\rm PML}$^+$ models. Models $\mathfrak{M}$ and $\mathfrak{M}'$ are {\rm PML}$^+$-equiv\-a\-lent if and only if they are forward bisimilar.
\end{theorem}
It follows from the theorem that if the worlds $w$ and $w'$ are PML$^+$-equiv\-a\-lent, then they are forward bisimilar. Thus, we obtain the Theorem 2.24 from \cite{Blackburn2001}, p.~69.

\begin{theorem}[The Hennessy-Milner theorem for PML$^-$]\label{Hennessy-Milner Theorem for PML-}
Let $\mathfrak{M}=(W, R, V)$ and $\mathfrak{M}'=(W', R', V')$ be two do\-main-fi\-nite {\rm PML}$^-$ models. Models $\mathfrak{M}$ and $\mathfrak{M}'$ are {\rm PML}$^-$-equivalent if and only if they are backward bisimilar.
\end{theorem}

\begin{theorem}[The Hennessy-Milner theorem for PML]\label{Hennessy-Milner Theorem for PML}
Let $\mathfrak{M}=(W,\allowbreak R, V)$ and $\mathfrak{M}'=(W', R', V')$ be two degree-finite {\rm PML} models. Models $\mathfrak{M}$ and $\mathfrak{M}'$ are {\rm PML}-equivalent if and only if they are regular bisimilar.
\end{theorem}

Also, an analogous statement as the Remark \ref{remark - first Lemma does not hold for presimulations} holds in Propositional Modal Logic, i.e., the Hennessy-Milner property is not valid for simulations.

\section{Computational examples}\label{Section - Computational examples}
In this section, we give examples which demonstrate the application of the Hen\-nessy-Milner-type theorems from the previous sections and clarify the relationships between different types of strong and weak bisimulations.

It is generally known that every linearly ordered Heyting algebra is a G\" odel algebra and every G\" odel algebra is a Heyting algebra with the Dummett condition $(x\rightarrow y)\vee (y\rightarrow x)=1$. Therefore, several examples are on the standard G\" odel modal logic over $[0,1]$ while the last example is on the Boolean two-valued structure.

\begin{example}\label{Example - First example with two relations}
Let $\mathfrak{M}=(W, \{R_{i}\}_{i\in I}, V)$ and $\mathfrak{M}'=(W', \{R_{i}'\}_{i\in I}, V')$ be two fuzzy Kripke models over the G\" odel structure, where $W=\{u,v,w\}$, $W'=\{v', w'\}$ and set $I=\{1,2\}$. Fuzzy relations $R_1, R_2$, $R_1 ', R_2'$ and fuzzy sets $V_p$, $V_q$, $V_p '$ and $V_q'$ are represented by the following fuzzy matrices and column vectors:
\begin{align*}
&R_1=\begin{bmatrix}
   1 & 0.3 & 1 \\
   0.6 & 0.4 & 0.6 \\
   1 & 0.4 & 1
  \end{bmatrix}\mbox{,}\quad
  R_2=\begin{bmatrix}
   0.8 & 0.5 & 0.8 \\
   0.6 & 0 & 0.6 \\
   0.9 & 0.5 & 0.9
  \end{bmatrix}\mbox{,}\quad V_p=\begin{bmatrix}
    1 \\
    0.6\\
    1
  \end{bmatrix}\mbox{,}\quad V_q=\begin{bmatrix}
    1 \\
    0.3\\
    1
  \end{bmatrix}\mbox{,}\\
&R_1'=\begin{bmatrix}
   1 & 0.4 \\
   0.6 & 0.4 \\
\end{bmatrix}\mbox{,}\quad
R_2'=\begin{bmatrix}
   0.9 & 0.5 \\
   0.6 & 0.3 \\
\end{bmatrix}\mbox{,}\quad
V_p'=\begin{bmatrix}
    1 \\
    0.6
\end{bmatrix}\mbox{,}\quad V_q'=\begin{bmatrix}
    1 \\
    0.3
\end{bmatrix}\mbox{.}
\end{align*}
Using algorithms for testing the existence and computing the greatest (pre)bi\-sim\-u\-lations between fuzzy Kripke models $\mathfrak{M}$ and $\mathfrak{M}'$ from \cite{Stankovic2021}, we have:
\begin{align*}
&\varphi_{*}^{fb}=\begin{bmatrix}
   0.3 & 0.3 \\
   0.3 & 0.3 \\
   0.3 & 0.3
 \end{bmatrix}\mbox{,} \quad
\varphi_{*}^{bb}=\varphi^{bb}=\begin{bmatrix}
   1 & 0.3 \\
   0.3 & 1 \\
   1 & 0.3
  \end{bmatrix}\mbox{,}\\
&\varphi_{*}^{fbb}=\varphi^{fbb}=\begin{bmatrix}
   1 & 0.3 \\
   0.3 & 1 \\
   1 & 0.3
 \end{bmatrix}\mbox{,}\quad
  \varphi_{*}^{bfb}=\begin{bmatrix}
   0.3 & 0.3 \\
   0.3 & 0.3 \\
   0.3 & 0.3
  \end{bmatrix}\mbox{,}\quad
  \varphi_{*}^{rb}=\begin{bmatrix}
   0.3 & 0.3 \\
   0.3 & 0.3 \\
   0.3 & 0.3
  \end{bmatrix}\mbox{,}
\end{align*}
and $\varphi^{fb}_*$, $\varphi^{bfb}_*$ and $\varphi^{rb}_*$ do not satisfy ($fb$-1), ($bfb$-1) and ($rb$-1), respectively, which means that $\varphi^{fb}$, $\varphi^{bfb}$ and $\varphi^{rb}$ do not exist.

According to the Theorem \ref{Hennessy-Milner Theorem for minus-formulae} and Definition \ref{Definition - formulae-equivalent}, it follows that models $\mathfrak{M}$ and $\mathfrak{M}'$ are $\PhiIHminus$-equivalent.

If we consider the reverse fuzzy Kripke models $\mathfrak{M}^{-1}$ and $\mathfrak{M}'^{-1}$, we have the opposite situation. Namely, in this case there are no ${bb}$-, $fbb$- and $rb$-bisimulations. In this case, according to the Theorem \ref{Hennessy-Milner Theorem for plus-formulae}, and Definition \ref{Definition - formulae-equivalent} it follows that models $\mathfrak{M}^{-1}$ and $\mathfrak{M}'^{-1}$ are $\PhiIHplus$-equivalent.
\end{example}

\begin{example}\label{Example - Second example with two relations}
Let us replace fuzzy relations $R_1, R_2$, $R_1 ', R_2'$ and fuzzy sets $V_p$, $V_q$, $V_p ', V_q'$ in the previous example with the following fuzzy matrices and column vectors:
\begin{align*}
&R_1=\begin{bmatrix}
   0 & 0.2 & 0.2 \\
   1 & 0.4 & 1 \\
   0 & 0.2 & 0
  \end{bmatrix}\mbox{,}\quad
  R_2=\begin{bmatrix}
   1 & 0.9 & 0.9 \\
   0.8 & 0.7 & 0.8 \\
   0.9 & 0.9 & 1
  \end{bmatrix}\mbox{,}\quad V_p=\begin{bmatrix}
    0.7 \\
    0.4\\
    0.7
  \end{bmatrix}\mbox{,}\quad V_q=\begin{bmatrix}
    0.8 \\
    1\\
    0.8
  \end{bmatrix}\mbox{,}\\
&R_1'=\begin{bmatrix}
   0.2 & 0.2 \\
   1 & 0.4 \\
\end{bmatrix}\mbox{,}\quad
R_2'=\begin{bmatrix}
   1 & 0.9 \\
   0.8 & 0.7 \\
\end{bmatrix}\mbox{,}\quad
V_p'=\begin{bmatrix}
    0.7 \\
    0.4
\end{bmatrix}\mbox{,}\quad V_q'=\begin{bmatrix}
    0.8 \\
    1
\end{bmatrix}\mbox{.}
\end{align*}
Using algorithms for testing the existence and computing the greatest (pre)bi\-sim\-u\-lations between fuzzy Kripke models $\mathfrak{M}$ and $\mathfrak{M}'$ from \cite{Stankovic2021}, we have:
\begin{align*}
&\varphi^{fb}=\begin{bmatrix}
   1 & 0.2 \\
   0.2 & 1 \\
   1 & 0.2
 \end{bmatrix}\mbox{,} \quad
\varphi^{bb}=\begin{bmatrix}
   1 & 0.4 \\
   0.4 & 1 \\
   1 & 0.4
  \end{bmatrix}\mbox{,}\\
&\varphi^{fbb}=\begin{bmatrix}
   1 & 0.4 \\
   0.2 & 1 \\
   1 & 0.4
 \end{bmatrix}\mbox{,}\quad
  \varphi^{bfb}=\begin{bmatrix}
   1 & 0.2 \\
   0.4 & 1 \\
   1 & 0.2
  \end{bmatrix}\mbox{,}\quad
  \varphi^{rb}=\begin{bmatrix}
   1 & 0.2 \\
   0.2 & 1 \\
   1 & 0.2
  \end{bmatrix}\mbox{.}
\end{align*}
In this example, all prebisimulations $\varphi^\theta_{*}$ for $\theta \in \{fb, bb, fbb, bfb, rb\}$ satisfy the condition ($\theta$-1).

According to the Theorem \ref{Hennessy-Milner Theorem for all formulae} and Definition \ref{Definition - formulae-equivalent}, it follows that models $\mathfrak{M}$ and $\mathfrak{M}'$ are $\PhiIH$-equivalent. Clearly, these models are also $\PhiIHplus$-equivalent and $\PhiIHminus$-equivalent.
\end{example}

\begin{example}
If we recall Example \ref{Example - first example}, we can conclude that the existence of a forward bisimulation does not imply that the models $\mathfrak{M}$ and $\mathfrak{M}'$ are $\PhiIHplus$-equivalent. According to the Definition \ref{Definition - formulae-equivalent}, we can conclude that worlds $v$ and $v'$ are $\PhiIHplus$-equivalent.

The following conclusion can be drawn based on all the above: for models to be logically equivalent, the weak bisimulation for set $\Phi$ must have at least one element $1$ in each of the rows and columns.

The situation from the Example \ref{Example - first example} can be interpreted in the following way: ``models $\mathfrak{M}$ and $\mathfrak{M}'$ are as $\PhiIHplus$-equivalent as they are forward bisimilar and vice versa''.
\end{example}

The following example illustrates the situation where fuzzy Kripke models are restricted to crisp values $\{0,1\}$.
\begin{example}\label{Example - crisp models}
Let $\mathfrak{M}=(W, R, V)$ and $\mathfrak{M}'=(W', R', V')$ be two Kripke models over the two-valued Boolean structure, where $W=\{t,u,v,w\}$ and $W'=\{v', w'\}$. Relations $R$, $R'$ and propositional variables $V_p$, $V_q$, $V_p '$ and $V_q'$ are represented by the following matrices and column vectors:
\begin{align*}
&R=\begin{bmatrix}
   0 & 0 & 0 & 0 \\
   1 & 0 & 0 & 0 \\
   0 & 0 & 0 & 1 \\
   0 & 0 & 0 & 0
  \end{bmatrix}\mbox{,}\quad V_p=\begin{bmatrix}
    1 \\
    0 \\
    0 \\
    1
  \end{bmatrix}\mbox{,}\quad V_q=\begin{bmatrix}
    1 \\
    1 \\
    1 \\
    1
  \end{bmatrix}\mbox{,}\\
&R'=\begin{bmatrix}
   0 & 0 \\
   1 & 0 \\
\end{bmatrix}\mbox{,}\quad
V_p'=\begin{bmatrix}
    1 \\
    0
\end{bmatrix}\mbox{,}\quad V_q'=\begin{bmatrix}
    1 \\
    1
\end{bmatrix}\mbox{.}
\end{align*}
Using algorithms for testing the existence and computing the greatest (pre)bi\-sim\-u\-lations between Kripke models $\mathfrak{M}$ and $\mathfrak{M}'$ from \cite{Stankovic2021}, we have:
\begin{align*}
&\varphi^{fb}=\varphi^{bb}=\varphi^{fbb}=\varphi^{bfb}=\varphi^{rb}=\begin{bmatrix}
   1 & 0 \\
   0 & 1 \\
   0 & 1 \\
   1 & 0
  \end{bmatrix}\mbox{.}
\end{align*}
Again, all prebisimulations $\varphi^\theta_{*}$ for $\theta \in \{fb, bb, fbb, bfb, rb\}$ satisfy the condition ($\theta$-1).

According to the Theorem \ref{Hennessy-Milner Theorem for PML} and Definition \ref{Definition - formulae-equivalent}, it follows that models $\mathfrak{M}$ and $\mathfrak{M}'$ are {\rm PML}-equivalent. Clearly, these models are also {\rm PML}$^{+}$-equivalent and {\rm PML}$^{-}$-equivalent.
\end{example}

\section{Concluding Remarks}\label{Concluding Remarks}

In this paper, we have combined the ideas about simulations and bisimulations for the fuzzy Kripke models from our previous work \cite{Stankovic2021} and the ideas about the Hennessy-Milner property for G\"{o}del Modal Logics from the work of T.F.~Fan \cite{Fan2015} to get the new results. Inter alia, weak (pre)simulations and weak (pre)bisimulations have been defined and some of their properties explained. Furthermore, using weak bisimulations, we have examined the formulae equivalence between the fuzzy Kripke models. We have showed that a weak bisimulation for the set of the plus-formulae between two image-finite fuzzy Kripke models is equal to a forward bisimulation between them. Using the principle of duality, a weak bisimulation for the set of minus-formulae is equal to a backward bisimulation between domain-finite Kripke models. Finally, we have determined a self-dual assertion that a weak bisimulation for the set of all formulae was equal to a regular bisimulation between degree-finite Kripke models.

\section{Acknowledgements}
This research was supported by the Science Fund of the Republic of Serbia, \# GRANT No.~7750185, Quantitative Automata Models: Fundamental Problems and Applications - QUAM.

The authors would like to thank the anonymous referees for their constructive suggestions that have significantly improved the paper.


\end{document}